\documentclass{article}
\usepackage[utf8]{inputenc}
\usepackage[margin=1in]{geometry}
\usepackage{mathtools}
\usepackage{amsfonts}
\usepackage{amsmath}
\usepackage{amssymb}
\usepackage{amsthm}
\usepackage{easytable}
\usepackage[ruled,vlined]{algorithm2e}
\usepackage{xcolor}

\usepackage{hyperref}
\usepackage[english]{babel}

\hypersetup{
pdfencoding=auto,
psdextra,
breaklinks=true}

\newtheorem{lemma}{Lemma}
\newtheorem{theorem}{Theorem}
\newtheorem{proposition}{Proposition}
\newtheorem{corollary}{Corollary}
\newtheorem{conjecture}{Conjecture}

\usepackage{tikz}

\usetikzlibrary{arrows.meta}
\usetikzlibrary{decorations.markings}
\usetikzlibrary{automata,arrows}
\usetikzlibrary{positioning}
\usetikzlibrary{calc}
\usetikzlibrary{shadings}

\tikzset{
dot/.style = {circle, fill, minimum size=#1,
              inner sep=0pt, outer sep=0pt},
dot/.default = 4pt 
}

\newcommand{\bound}{\useasboundingbox (-2,-2) -- (2,2);}

\newcommand{\Ac}{\mathcal{A}}

\newcommand{\Sc}{\mathcal{S}}

\renewcommand{\vert}[2]{\coordinate (#1) at (#2); \node[dot] at (#2) {};}
\newcommand{\labvert}[3]{\coordinate (#1) at (#2); \node[dot] at (#2) {}; \node[#3] at (#2) {#1};}
\newcommand{\edge}[2]{\draw (#1) -- (#2);}

\DeclarePairedDelimiter{\ceil}{\lceil}{\rceil}

\DeclarePairedDelimiter{\abs}{\lvert}{\rvert}
\DeclareMathOperator{\dom}{dom}

\title{Hadwiger's Conjecture with Certain Forbidden Induced Subgraphs}
\author{Daniel Carter}
\date{}

\begin{document}

\maketitle

\begin{abstract}
    We prove that $\{\overline{K_3}, H\}$-free graphs are not counterexamples to Hadwiger's Conjecture, where $H$ is any one of 33 graphs on seven, eight, or nine vertices, or $H=K_8$. This improves on past results of Plummer-Stiebitz-Toft, Kriesell, and Bosse. The proofs are mostly computer-assisted.
\end{abstract}

\section{Introduction}

All graphs in this paper are finite and simple. Let $h(G)$ be the maximum $t$ such that $K_t$ is a minor of $G$, $\chi(G)$ the chromatic number of $G$, and $\omega(G)$ and $\alpha(G)$ respectively the clique number and independence number of $G$. \textit{Hadwiger's Conjecture} is:

\begin{conjecture}
For all graphs $G$, $h(G)\ge \chi(G)$.
\end{conjecture}

For a graph $H$, we say $G$ is \textit{$H$-free} if no induced subgraph of $G$ is isomorphic to $H$. Likewise for a set of graphs $S$, we say $G$ is \textit{$S$-free} if $G$ is $H$-free for all $H$ in $S$. We denote by HC-$S$ the statement that Hadwiger's Conjecture holds for all $S$-free graphs.

There is much interest in the specific case of Hadwiger's Conjecture for graphs with no stable set of size 3, i.e. $\overline{K_3}$-free graphs. Paul Seymour writes about this case:
\begin{quote}
    This seems to me to be an excellent place to look for a counterexample. My own belief is, if it is true for graphs with stability number two then it is probably true in general. \cite{seymoursurvey}
\end{quote}

Previously, it was proved:

\begin{theorem}\label{thm:past}
HC-$\{\overline{K_3},H\}$ holds where $H$ is any graph on five vertices or the graphs $H_6$ or $H_7$ in Figure~\ref{fig:h6h7} \cite{pst, kriesell}, or the wheel graph on six vertices $W_5$, $\overline{K_{1,5}}$, or $K_7$ \cite{bosse}.
\end{theorem}

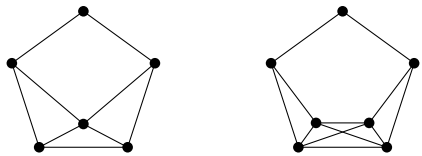
\begin{figure}[htb!]
    \centering
    \begin{tikzpicture}
    \vert{c1}{0,1}
    \vert{c2}{-0.951,0.309}
    \vert{c3}{-0.588,-0.809}
    \vert{c4}{0.588,-0.809}
    \vert{c5}{0.951,0.309}
    \edge{c1}{c2}
    \edge{c2}{c3}
    \edge{c3}{c4}
    \edge{c4}{c5}
    \edge{c5}{c1}
    
    \vert{s1}{0,-0.5}
    \edge{s1}{c2}
    \edge{s1}{c3}
    \edge{s1}{c4}
    \edge{s1}{c5}
    \end{tikzpicture}\qquad\qquad
    \begin{tikzpicture}
    \vert{c1}{0,1}
    \vert{c2}{-0.951,0.309}
    \vert{c3}{-0.588,-0.809}
    \vert{c4}{0.588,-0.809}
    \vert{c5}{0.951,0.309}
    \edge{c1}{c2}
    \edge{c2}{c3}
    \edge{c3}{c4}
    \edge{c4}{c5}
    \edge{c5}{c1}
    
    \vert{x3}{-0.353,-0.485}
    \edge{x3}{c2}
    \edge{x3}{c3}
    \edge{x3}{c4}
    
    \vert{x4}{0.353,-0.485}
    \edge{x4}{c3}
    \edge{x4}{c4}
    \edge{x4}{c5}
    \edge{x3}{x4}
    \end{tikzpicture}
    \caption{The graphs $H_6$ (left) and $H_7$ (right).}
    \label{fig:h6h7}
\end{figure}

We improve these results using three main facts:

\begin{lemma}\label{lem:dominating}\cite{pst}
Any minimal counterexample to HC-$\{\overline{K_3}\}$ has no dominating edges.
\end{lemma}
Here, a \textit{dominating edge} is an edge $uv$ such that all vertices are adjacent to $u$ or $v$ (or both), and by \textit{minimal counterexample} we mean that no proper induced subgraph is also a counterexample.

\begin{lemma}\label{lem:fourclique}\cite{packingseagulls}
Any $\overline{K_3}$-free graph with
\[ \omega(G) \ge \begin{cases}
\ceil{\frac{\abs{G}}{4}} & \text{if $\abs{G}$ is even,} \\
\ceil{\frac{\abs{G}+3}{4}} & \text{if $\abs{G}$ is odd}
\end{cases} \]
has $h(G)\ge \chi(G)$.
\end{lemma}

\begin{lemma}\label{lem:conings}
$\{K_3, H\}$-free graphs have chromatic number at most $k$, where $(H, k)$ is any of the pairs $(K_1, 0)$, $(K_2, 1)$, $(P_4, 2)$, or $(T_1, 3)$, $(T_2, 3)$, or $(T_3, 3)$ \cite{colorablepairs},\footnote{In fact, $(T_1', 3)$ is also such a pair where $T_1'$ is $T_1$ with a second leaf added to the third vertex in the 5-path \cite{trianglefork}. The author was not aware of this stronger result until after finding the main theorem of this paper. Replacing $T_1$ by $T_1'$ in the relevant algorithm could potentially lead to slightly better results.} where $T_i$ are shown in Figure~\ref{fig:trees}.
\end{lemma}

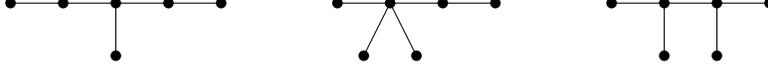
\begin{figure}
    \centering
    \begin{tikzpicture}[scale=0.7]
    \vert{1}{-2,0.5};
    \vert{2}{-1,0.5};
    \vert{3}{0,0.5};
    \vert{4}{1,0.5};
    \vert{5}{2,0.5};
    \vert{6}{0,-0.5};
    \edge{1}{2};
    \edge{2}{3};
    \edge{3}{4};
    \edge{4}{5};
    \edge{3}{6};
    \end{tikzpicture}\qquad\qquad
    \begin{tikzpicture}[scale=0.7]
    \vert{1}{-1.5,0.5};
    \vert{2}{-0.5,0.5};
    \vert{3}{0.5,0.5};
    \vert{4}{1.5,0.5};
    \vert{5}{-1,-0.5};
    \vert{6}{-0,-0.5};
    \edge{1}{2};
    \edge{2}{3};
    \edge{3}{4};
    \edge{2}{5};
    \edge{2}{6};
    \end{tikzpicture}\qquad\qquad
    \begin{tikzpicture}[scale=0.7]
    \vert{1}{-1.5,0.5};
    \vert{2}{-0.5,0.5};
    \vert{3}{0.5,0.5};
    \vert{4}{1.5,0.5};
    \vert{5}{-0.5,-0.5};
    \vert{6}{0.5,-0.5};
    \edge{1}{2};
    \edge{2}{3};
    \edge{3}{4};
    \edge{2}{5};
    \edge{3}{6};
    \end{tikzpicture}
    \caption{The trees $T_1$, $T_2$, and $T_3$ (left-to-right) associated with Lemma~\ref{lem:conings}.}
    \label{fig:trees}
\end{figure}

By starting from a graph that is known to be an induced subgraph of any counterexample to HC-$\{\overline{K_3}\}$ and iteratively applying these lemmas, it is sometimes possible to find a new graph that must be an induced subgraph of any counterexample to HC-$\{\overline{K_3}\}$. Specifically, Lemma~\ref{lem:dominating} implies, given a dominating edge $uv$, the existence of a vertex adjacent to neither $u$ or $v$, which can attach onto the graph in some number of ways. If there are no dominating edges, we can use the other two lemmas to get the existence of one or more vertices that attach onto the graph in such a way as to introduce new dominating edges.

In particular, we prove:

\begin{theorem}\label{thm:main}
HC-$\{\overline{K_3}, H\}$ holds where $H\in\{H_1',H_2',\dots,H_{33}'\}$ is any of the graphs in Figure~\ref{fig:main}.
\end{theorem}

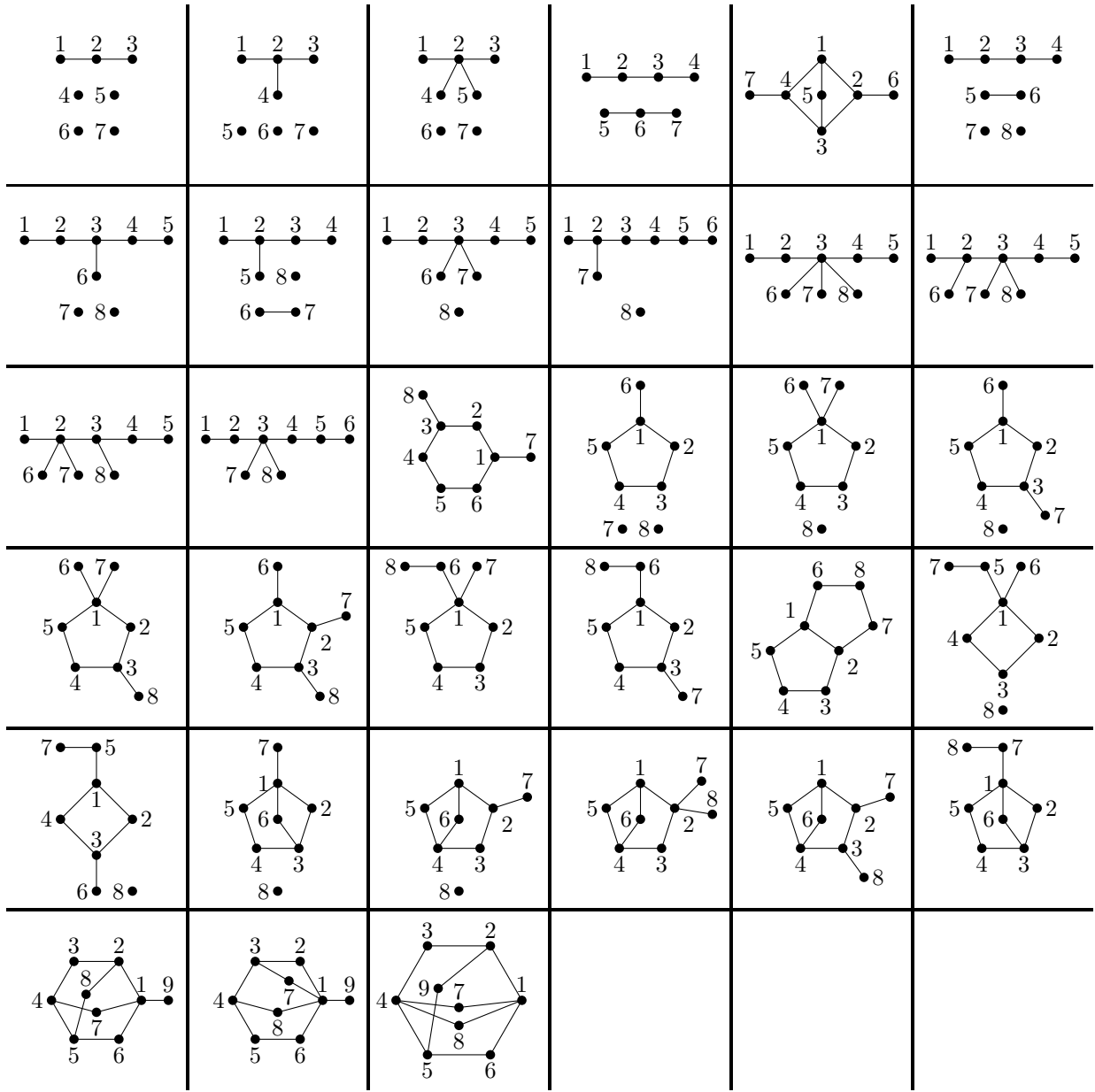
\begin{figure}[htbp!]
    \centering
    \begin{TAB}(e)[0.75em]{c|c|c|c|c|c}{c|c|c|c|c|c}
    \begin{tikzpicture}[scale=0.54]
    \bound;
    \labvert{1}{-1,1}{above};
    \labvert{2}{0,1}{above};
    \labvert{3}{1,1}{above};
    \labvert{4}{-0.5,0}{left};
    \labvert{5}{0.5,0}{left};
    \labvert{6}{-0.5,-1}{left};
    \labvert{7}{0.5,-1}{left};
    \edge{1}{2};
    \edge{2}{3};
    \end{tikzpicture} &
    \begin{tikzpicture}[scale=0.54]
    \bound;
    \labvert{1}{-1,1}{above};
    \labvert{2}{0,1}{above};
    \labvert{3}{1,1}{above};
    \labvert{4}{0,0}{left};
    \labvert{5}{-1,-1}{left};
    \labvert{6}{0,-1}{left};
    \labvert{7}{1,-1}{left};
    \edge{1}{2};
    \edge{2}{3};
    \edge{2}{4};
    \end{tikzpicture} &
    \begin{tikzpicture}[scale=0.54]
    \bound;
    \labvert{1}{-1,1}{above};
    \labvert{2}{0,1}{above};
    \labvert{3}{1,1}{above};
    \labvert{4}{-0.5,0}{left};
    \labvert{5}{0.5,0}{left};
    \labvert{6}{-0.5,-1}{left};
    \labvert{7}{0.5,-1}{left};
    \edge{1}{2};
    \edge{2}{3};
    \edge{2}{4};
    \edge{2}{5};
    \end{tikzpicture} &
    \begin{tikzpicture}[scale=0.54]
    \bound;
    \labvert{1}{-1.5,0.5}{above};
    \labvert{2}{-0.5,0.5}{above};
    \labvert{3}{0.5,0.5}{above};
    \labvert{4}{1.5,0.5}{above};
    \labvert{5}{-1,-0.5}{below};
    \labvert{6}{0,-0.5}{below};
    \labvert{7}{1,-0.5}{below};
    \edge{1}{2};
    \edge{2}{3};
    \edge{3}{4};
    \edge{5}{6};
    \edge{6}{7};
    \end{tikzpicture} &
    \begin{tikzpicture}[scale=0.54]
    \bound;
    \labvert{1}{0,1}{above};
    \labvert{2}{1,0}{above};
    \labvert{3}{0,-1}{below};
    \labvert{4}{-1,0}{above};
    \labvert{5}{0,0}{left};
    \labvert{6}{2,0}{above};
    \labvert{7}{-2,0}{above};
    \edge{1}{2};
    \edge{2}{3};
    \edge{3}{4};
    \edge{4}{1};
    \edge{1}{5};
    \edge{3}{5};
    \edge{2}{6};
    \edge{4}{7};
    \end{tikzpicture} &
    \begin{tikzpicture}[scale=0.54]
    \bound;
    \labvert{1}{-1.5,1}{above};
    \labvert{2}{-0.5,1}{above};
    \labvert{3}{0.5,1}{above};
    \labvert{4}{1.5,1}{above};
    \labvert{5}{-0.5,0}{left};
    \labvert{6}{0.5,0}{right};
    \labvert{7}{-0.5,-1}{left};
    \labvert{8}{0.5,-1}{left};
    \edge{1}{2};
    \edge{2}{3};
    \edge{3}{4};
    \edge{5}{6};
    \end{tikzpicture} \\
    \begin{tikzpicture}[scale=0.54]
    \bound;
    \labvert{1}{-2,1}{above};
    \labvert{2}{-1,1}{above};
    \labvert{3}{0,1}{above};
    \labvert{4}{1,1}{above};
    \labvert{5}{2,1}{above};
    \labvert{6}{0,0}{left};
    \labvert{7}{-0.5,-1}{left};
    \labvert{8}{0.5,-1}{left};
    \edge{1}{2};
    \edge{2}{3};
    \edge{3}{4};
    \edge{4}{5};
    \edge{3}{6};
    \end{tikzpicture} &
    \begin{tikzpicture}[scale=0.54]
    \bound;
    \labvert{1}{-1.5,1}{above};
    \labvert{2}{-0.5,1}{above};
    \labvert{3}{0.5,1}{above};
    \labvert{4}{1.5,1}{above};
    \labvert{5}{-0.5,0}{left};
    \labvert{8}{0.5,0}{left};
    \labvert{6}{-0.5,-1}{left};
    \labvert{7}{0.5,-1}{right};
    \edge{1}{2};
    \edge{2}{3};
    \edge{3}{4};
    \edge{2}{5};
    \edge{6}{7};
    \end{tikzpicture} &
    \begin{tikzpicture}[scale=0.54]
    \bound;
    \labvert{1}{-2,1}{above};
    \labvert{2}{-1,1}{above};
    \labvert{3}{0,1}{above};
    \labvert{4}{1,1}{above};
    \labvert{5}{2,1}{above};
    \labvert{6}{-0.5,0}{left};
    \labvert{7}{0.5,0}{left};
    \labvert{8}{0,-1}{left};
    \edge{1}{2};
    \edge{2}{3};
    \edge{3}{4};
    \edge{4}{5};
    \edge{3}{6};
    \edge{3}{7};
    \end{tikzpicture} &
    \begin{tikzpicture}[scale=0.54]
    \bound;
    \labvert{1}{-2,1}{above};
    \labvert{2}{-1.2,1}{above};
    \labvert{3}{-0.4,1}{above};
    \labvert{4}{0.4,1}{above};
    \labvert{5}{1.2,1}{above};
    \labvert{6}{2,1}{above};
    \labvert{7}{-1.2,0}{left};
    \labvert{8}{0,-1}{left};
    \edge{1}{2};
    \edge{2}{3};
    \edge{3}{4};
    \edge{4}{5};
    \edge{5}{6};
    \edge{2}{7};
    \end{tikzpicture} &
    \begin{tikzpicture}[scale=0.54]
    \bound;
    \labvert{1}{-2,0.5}{above};
    \labvert{2}{-1,0.5}{above};
    \labvert{3}{0,0.5}{above};
    \labvert{4}{1,0.5}{above};
    \labvert{5}{2,0.5}{above};
    \labvert{6}{-1,-0.5}{left};
    \labvert{7}{0,-0.5}{left};
    \labvert{8}{1,-0.5}{left};
    \edge{1}{2};
    \edge{2}{3};
    \edge{3}{4};
    \edge{4}{5};
    \edge{3}{6};
    \edge{3}{7};
    \edge{3}{8};
    \end{tikzpicture} &
    \begin{tikzpicture}[scale=0.54]
    \bound;
    \labvert{1}{-2,0.5}{above};
    \labvert{2}{-1,0.5}{above};
    \labvert{3}{0,0.5}{above};
    \labvert{4}{1,0.5}{above};
    \labvert{5}{2,0.5}{above};
    \labvert{6}{-1.5,-0.5}{left};
    \labvert{7}{-0.5,-0.5}{left};
    \labvert{8}{0.5,-0.5}{left};
    \edge{1}{2};
    \edge{2}{3};
    \edge{3}{4};
    \edge{4}{5};
    \edge{2}{6};
    \edge{3}{7};
    \edge{3}{8};
    \end{tikzpicture} \\
    \begin{tikzpicture}[scale=0.54]
    \bound;
    \labvert{1}{-2,0.5}{above};
    \labvert{2}{-1,0.5}{above};
    \labvert{3}{0,0.5}{above};
    \labvert{4}{1,0.5}{above};
    \labvert{5}{2,0.5}{above};
    \labvert{6}{-1.5,-0.5}{left};
    \labvert{7}{-0.5,-0.5}{left};
    \labvert{8}{0.5,-0.5}{left};
    \edge{1}{2};
    \edge{2}{3};
    \edge{3}{4};
    \edge{4}{5};
    \edge{2}{6};
    \edge{2}{7};
    \edge{3}{8};
    \end{tikzpicture} &
    \begin{tikzpicture}[scale=0.54]
    \bound;
    \labvert{1}{-2,0.5}{above};
    \labvert{2}{-1.2,0.5}{above};
    \labvert{3}{-0.4,0.5}{above};
    \labvert{4}{0.4,0.5}{above};
    \labvert{5}{1.2,0.5}{above};
    \labvert{6}{2,0.5}{above};
    \labvert{7}{-0.9,-0.5}{left};
    \labvert{8}{0.1,-0.5}{left};
    \edge{1}{2};
    \edge{2}{3};
    \edge{3}{4};
    \edge{4}{5};
    \edge{5}{6};
    \edge{3}{7};
    \edge{3}{8};
    \end{tikzpicture} &
    \begin{tikzpicture}[scale=0.54]
    \bound;
    \labvert{1}{1,0}{left};
    \labvert{2}{0.5,0.866}{above};
    \labvert{3}{-0.5,0.866}{left};
    \labvert{4}{-1,0}{left};
    \labvert{5}{-0.5,-0.866}{below};
    \labvert{6}{0.5,-0.866}{below};
    \labvert{7}{2,0}{above};
    \labvert{8}{-1,1.732}{left};
    \edge{1}{2};
    \edge{2}{3};
    \edge{3}{4};
    \edge{4}{5};
    \edge{5}{6};
    \edge{6}{1};
    \edge{1}{7};
    \edge{3}{8};
    \end{tikzpicture} &
    \begin{tikzpicture}[scale=0.54]
    \bound;
    \labvert{1}{0,1}{below};
    \labvert{2}{0.951,0.309}{right};
    \labvert{3}{0.588,-0.809}{below};
    \labvert{4}{-0.588,-0.809}{below};
    \labvert{5}{-0.951,0.309}{left};
    \labvert{6}{0,2}{left};
    \labvert{7}{-0.5,-2}{left};
    \labvert{8}{0.5,-2}{left};
    \edge{1}{2};
    \edge{2}{3};
    \edge{3}{4};
    \edge{4}{5};
    \edge{5}{1};
    \edge{1}{6};
    \end{tikzpicture} &
    \begin{tikzpicture}[scale=0.54]
    \bound;
    \labvert{1}{0,1}{below};
    \labvert{2}{0.951,0.309}{right};
    \labvert{3}{0.588,-0.809}{below};
    \labvert{4}{-0.588,-0.809}{below};
    \labvert{5}{-0.951,0.309}{left};
    \labvert{6}{-0.5,2}{left};
    \labvert{7}{0.5,2}{left};
    \labvert{8}{0,-2}{left};
    \edge{1}{2};
    \edge{2}{3};
    \edge{3}{4};
    \edge{4}{5};
    \edge{5}{1};
    \edge{1}{6};
    \edge{1}{7};
    \end{tikzpicture} &
    \begin{tikzpicture}[scale=0.54]
    \bound;
    \labvert{1}{0,1}{below};
    \labvert{2}{0.951,0.309}{right};
    \labvert{3}{0.588,-0.809}{right};
    \labvert{4}{-0.588,-0.809}{below};
    \labvert{5}{-0.951,0.309}{left};
    \labvert{6}{0,2}{left};
    \labvert{7}{1.176,-1.618}{right};
    \labvert{8}{0,-2}{left};
    \edge{1}{2};
    \edge{2}{3};
    \edge{3}{4};
    \edge{4}{5};
    \edge{5}{1};
    \edge{1}{6};
    \edge{3}{7};
    \end{tikzpicture} \\
    \begin{tikzpicture}[scale=0.54]
    \bound;
    \labvert{1}{0,1}{below};
    \labvert{2}{0.951,0.309}{right};
    \labvert{3}{0.588,-0.809}{right};
    \labvert{4}{-0.588,-0.809}{below};
    \labvert{5}{-0.951,0.309}{left};
    \labvert{6}{-0.5,2}{left};
    \labvert{7}{0.5,2}{left};
    \labvert{8}{1.176,-1.618}{right};
    \edge{1}{2};
    \edge{2}{3};
    \edge{3}{4};
    \edge{4}{5};
    \edge{5}{1};
    \edge{1}{6};
    \edge{1}{7};
    \edge{3}{8};
    \end{tikzpicture} &
    \begin{tikzpicture}[scale=0.54]
    \bound;
    \labvert{1}{0,1}{below};
    \labvert{2}{0.951,0.309}{below right};
    \labvert{3}{0.588,-0.809}{right};
    \labvert{4}{-0.588,-0.809}{below};
    \labvert{5}{-0.951,0.309}{left};
    \labvert{6}{0,2}{left};
    \labvert{7}{1.902,0.618}{above};
    \labvert{8}{1.176,-1.618}{right};
    \edge{1}{2};
    \edge{2}{3};
    \edge{3}{4};
    \edge{4}{5};
    \edge{5}{1};
    \edge{1}{6};
    \edge{2}{7};
    \edge{3}{8};
    \end{tikzpicture} &
    \begin{tikzpicture}[scale=0.54]
    \bound;
    \labvert{1}{0,1}{below};
    \labvert{2}{0.951,0.309}{right};
    \labvert{3}{0.588,-0.809}{below};
    \labvert{4}{-0.588,-0.809}{below};
    \labvert{5}{-0.951,0.309}{left};
    \labvert{6}{-0.5,2}{right};
    \labvert{7}{0.5,2}{right};
    \labvert{8}{-1.5,2}{left};
    \edge{1}{2};
    \edge{2}{3};
    \edge{3}{4};
    \edge{4}{5};
    \edge{5}{1};
    \edge{1}{6};
    \edge{1}{7};
    \edge{6}{8};
    \end{tikzpicture} &
    \begin{tikzpicture}[scale=0.54]
    \bound;
    \labvert{1}{0,1}{below};
    \labvert{2}{0.951,0.309}{right};
    \labvert{3}{0.588,-0.809}{right};
    \labvert{4}{-0.588,-0.809}{below};
    \labvert{5}{-0.951,0.309}{left};
    \labvert{6}{0,2}{right};
    \labvert{7}{1.176,-1.618}{right};
    \labvert{8}{-1,2}{left};
    \edge{1}{2};
    \edge{2}{3};
    \edge{3}{4};
    \edge{4}{5};
    \edge{5}{1};
    \edge{1}{6};
    \edge{3}{7};
    \edge{6}{8};
    \end{tikzpicture} &
    \begin{tikzpicture}[scale=0.54]
    \bound;
    \labvert{1}{-0.476,0.345}{above left};
    \labvert{2}{0.476,-0.345}{below right};
    \labvert{3}{0.112,-1.464}{below};
    \labvert{4}{-1.063,-1.464}{below};
    \labvert{5}{-1.427,-0.345}{left};
    \labvert{6}{-0.112,1.464}{above};
    \labvert{7}{1.427,0.345}{right};
    \labvert{8}{1.063,1.464}{above};
    \edge{1}{2};
    \edge{2}{3};
    \edge{3}{4};
    \edge{4}{5};
    \edge{5}{1};
    \edge{1}{6};
    \edge{2}{7};
    \edge{6}{8};
    \edge{7}{8};
    \end{tikzpicture} &
    \begin{tikzpicture}[scale=0.54]
    \bound;
    \labvert{1}{0,1}{below};
    \labvert{2}{1,0}{right};
    \labvert{3}{0,-1}{below};
    \labvert{4}{-1,0}{left};
    \labvert{5}{-0.5,2}{right};
    \labvert{6}{0.5,2}{right};
    \labvert{7}{-1.5,2}{left};
    \labvert{8}{0,-2}{left};
    \edge{1}{2};
    \edge{2}{3};
    \edge{3}{4};
    \edge{4}{1};
    \edge{1}{5};
    \edge{1}{6};
    \edge{5}{7};
    \end{tikzpicture} \\
    \begin{tikzpicture}[scale=0.54]
    \bound;
    \labvert{1}{0,1}{below};
    \labvert{2}{1,0}{right};
    \labvert{3}{0,-1}{above};
    \labvert{4}{-1,0}{left};
    \labvert{5}{0,2}{right};
    \labvert{6}{0,-2}{left};
    \labvert{7}{-1,2}{left};
    \labvert{8}{1,-2}{left};
    \edge{1}{2};
    \edge{2}{3};
    \edge{3}{4};
    \edge{4}{1};
    \edge{1}{5};
    \edge{3}{6};
    \edge{5}{7};
    \end{tikzpicture} &
    \begin{tikzpicture}[scale=0.54]
    \bound;
    \labvert{1}{0,1}{left};
    \labvert{2}{0.951,0.309}{right};
    \labvert{3}{0.588,-0.809}{below};
    \labvert{4}{-0.588,-0.809}{below};
    \labvert{5}{-0.951,0.309}{left};
    \labvert{6}{0,0}{left};
    \labvert{7}{0,2}{left};
    \labvert{8}{0,-2}{left};
    \edge{1}{2};
    \edge{2}{3};
    \edge{3}{4};
    \edge{4}{5};
    \edge{5}{1};
    \edge{1}{6};
    \edge{3}{6};
    \edge{1}{7};
    \end{tikzpicture} &
    \begin{tikzpicture}[scale=0.54]
    \bound;
    \labvert{1}{0,1}{above};
    \labvert{2}{0.951,0.309}{below right};
    \labvert{3}{0.588,-0.809}{below};
    \labvert{4}{-0.588,-0.809}{below};
    \labvert{5}{-0.951,0.309}{left};
    \labvert{6}{0,0}{left};
    \labvert{7}{1.902,0.618}{above};
    \labvert{8}{0,-2}{left};
    \edge{1}{2};
    \edge{2}{3};
    \edge{3}{4};
    \edge{4}{5};
    \edge{5}{1};
    \edge{1}{6};
    \edge{4}{6};
    \edge{2}{7};
    \end{tikzpicture} &
    \begin{tikzpicture}[scale=0.54]
    \bound;
    \labvert{1}{0,1}{above};
    \labvert{2}{0.951,0.309}{below right};
    \labvert{3}{0.588,-0.809}{below};
    \labvert{4}{-0.588,-0.809}{below};
    \labvert{5}{-0.951,0.309}{left};
    \labvert{6}{0,0}{left};
    \labvert{7}{1.699,1.063}{above};
    \labvert{8}{2,0.139}{above};
    \edge{1}{2};
    \edge{2}{3};
    \edge{3}{4};
    \edge{4}{5};
    \edge{5}{1};
    \edge{1}{6};
    \edge{4}{6};
    \edge{2}{7};
    \edge{2}{8};
    \end{tikzpicture} &
    \begin{tikzpicture}[scale=0.54]
    \bound;
    \labvert{1}{0,1}{above};
    \labvert{2}{0.951,0.309}{below right};
    \labvert{3}{0.588,-0.809}{right};
    \labvert{4}{-0.588,-0.809}{below};
    \labvert{5}{-0.951,0.309}{left};
    \labvert{6}{0,0}{left};
    \labvert{7}{1.902,0.618}{above};
    \labvert{8}{1.176,-1.618}{right};
    \edge{1}{2};
    \edge{2}{3};
    \edge{3}{4};
    \edge{4}{5};
    \edge{5}{1};
    \edge{1}{6};
    \edge{4}{6};
    \edge{2}{7};
    \edge{3}{8};
    \end{tikzpicture} &
    \begin{tikzpicture}[scale=0.54]
    \bound;
    \labvert{1}{0,1}{left};
    \labvert{2}{0.951,0.309}{right};
    \labvert{3}{0.588,-0.809}{below};
    \labvert{4}{-0.588,-0.809}{below};
    \labvert{5}{-0.951,0.309}{left};
    \labvert{6}{0,0}{left};
    \labvert{7}{0,2}{right};
    \labvert{8}{-1,2}{left};
    \edge{1}{2};
    \edge{2}{3};
    \edge{3}{4};
    \edge{4}{5};
    \edge{5}{1};
    \edge{1}{6};
    \edge{3}{6};
    \edge{1}{7};
    \edge{7}{8};
    \end{tikzpicture} \\
    \begin{tikzpicture}[scale=0.54]
    \bound;
    \labvert{1}{1.25,0}{above};
    \labvert{2}{0.625,1.083}{above};
    \labvert{3}{-0.625,1.083}{above};
    \labvert{4}{-1.25,0}{left};
    \labvert{5}{-0.625,-1.083}{below};
    \labvert{6}{0.625,-1.083}{below};
    \labvert{7}{0,-0.333}{below};
    \labvert{8}{-0.289,0.167}{above};
    \labvert{9}{2,0}{above};
    \edge{1}{2};
    \edge{2}{3};
    \edge{3}{4};
    \edge{4}{5};
    \edge{5}{6};
    \edge{6}{1};
    \edge{1}{7};
    \edge{4}{7};
    \edge{2}{8};
    \edge{5}{8};
    \edge{1}{9};
    \end{tikzpicture} &
    \begin{tikzpicture}[scale=0.54]
    \bound;
    \labvert{1}{1.25,0}{above};
    \labvert{2}{0.625,1.083}{above};
    \labvert{3}{-0.625,1.083}{above};
    \labvert{4}{-1.25,0}{left};
    \labvert{5}{-0.625,-1.083}{below};
    \labvert{6}{0.625,-1.083}{below};
    \labvert{8}{0,-0.333}{below};
    \labvert{7}{0.313,0.541}{below};
    \labvert{9}{2,0}{above};
    \edge{1}{2};
    \edge{2}{3};
    \edge{3}{4};
    \edge{4}{5};
    \edge{5}{6};
    \edge{6}{1};
    \edge{1}{8};
    \edge{4}{8};
    \edge{1}{7};
    \edge{3}{7};
    \edge{1}{9};
    \end{tikzpicture} &
    \begin{tikzpicture}[scale=0.54]
    \bound;
    \labvert{1}{1.75,0}{above};
    \labvert{2}{0.875,1.516}{above};
    \labvert{3}{-0.875,1.516}{above};
    \labvert{4}{-1.75,0}{left};
    \labvert{5}{-0.875,-1.516}{below};
    \labvert{6}{0.875,-1.516}{below};
    \labvert{7}{0,-0.2}{above};
    \labvert{9}{-0.577,0.333}{left};
    \labvert{8}{0,-0.7}{below};
    \edge{1}{2};
    \edge{2}{3};
    \edge{3}{4};
    \edge{4}{5};
    \edge{5}{6};
    \edge{6}{1};
    \edge{1}{7};
    \edge{4}{7};
    \edge{2}{9};
    \edge{5}{9};
    \edge{1}{8};
    \edge{4}{8};
    \end{tikzpicture} & & &
    \end{TAB}
    \caption{The complements of the graphs $H_1', \dots, H_{33}'$, when read from left-to-right, top-to-bottom. We display the complements because these graphs have many edges, so it is almost always easier to recognize them from their complements. The graphs are first ordered by number of vertices, then by common structural elements. The labelings are used to choose edges in a consistent manner for algorithms.}
    \label{fig:main}
\end{figure}

A small number of these cases are easily proved by hand, but most require hundreds of steps and are proved with computer assistance. We describe the relevant algorithms, prove their correctness, and prove some facts about their behavior and limitations in Section~\ref{sec:algorithms}. We then provide the data necessary to verify Theorem~\ref{thm:main} in Section~\ref{sec:data}, and we prove a few of the simple cases by hand for demonstration. We make some final remarks in Section~\ref{sec:conclusion}.

\subsection{Relationship to Past Results}

Theorem~\ref{thm:main} improves dramatically on Theorem~\ref{thm:past}. In fact, the only graphs in the statement of Theorem~\ref{thm:past} that are not an induced subgraph of any graph in the statement of Theorem~\ref{thm:main} are $H_7$ and $K_7$.

By Proposition~\ref{prop:complete}, if HC-$\{\overline{K_3}, K_8\}$ holds then Algorithm~\ref{alg:full} can prove this fact, though it may take an extremely long time. In fact, since the Ramsey number $R(3,8)=28$, the algorithm will most likely have to consider 27-vertex graphs at some point. Each such graph accrues a ``weight'' of $2^{27}=134217728$ (see Section~\ref{sec:conclusion}), which is already nearly 6 times the highest total weight of any of the proof steps in Figure~\ref{fig:data} used to prove Theorem~\ref{thm:main}, and there are 477142 $\{\overline{K_3},K_8\}$-free graphs with 27 vertices \cite{ramsey}, so one should not expect to improve $K_7$ to $K_8$ by running Algorithm~\ref{alg:full} for any reasonable amount of time. Instead, one should just check these 477142 graphs directly for counterexamples to HC-$\{\overline{K_3}\}$, since all graphs with 26 or fewer vertices have large enough clique number to satisfy the hypothesis of Lemma~\ref{lem:fourclique} due to the exact known values of the Ramsey numbers $R(3,k)$ with $k\le 7$. Our initial hope was that Theorem~\ref{thm:main} would make this check very easy, but this could not be further from the truth. In fact, \textit{all} $\{\overline{K_3},K_8\}$-free graphs on 27 vertices contain \textit{all} of the $H_i'$ (and $H_7$), so Theorem~\ref{thm:main} is no help at all. Instead, we perform a check using other means in Appendix~\ref{app:k8}, and get:
\begin{theorem}\label{thm:k8}
HC-$\{\overline{K_3},K_8\}$ holds.
\end{theorem}
and also
\begin{corollary}\label{cor:k8bound}
A minimal counterexample to HC-$\{\overline{K_3}\}$ has either 31 vertices or at least 33 vertices.
\end{corollary}
\begin{proof}
By the theorem above, a minimal counterexample must have a $K_8$. If such a graph has fewer than 31 vertices or has 32 vertices exactly, Lemma~\ref{lem:fourclique} implies it is not a counterexample to HC-$\{\overline{K_3}\}$.
\end{proof}
Since $R(3,9)=36$, it seems computationally infeasible to improve $K_8$ to $K_9$ by brute-force.

The fact that $H_7$ was not improved to a larger graph is due primarily to the combination of the facts that the ``$\overline{C_4}$-core'' of $H_7$ is all of $H_7$ (see Section~\ref{sec:limitations}), and $H_7$ has no dominating edges. The only way $H_7$ could be improved by the algorithms in this paper is directly starting from $H_7$ itself, but this is a poor starting place given the lack of dominating edges.

\section{Algorithms}
\label{sec:algorithms}

\subsection{Basic Algorithm: No Dominating Edges}

Suppose all counterexamples to HC-$\{\overline{K_3}\}$ are known to contain induced subgraph $H$, i.e. HC-$\{\overline{K_3}, H\}$ holds. If $H$ has a dominating edge $uv$, then by Lemma~\ref{lem:dominating}, all minimal counterexamples to HC-$\{\overline{K_3}\}$ have an induced $H$ plus a vertex $w$ adjacent to neither $u$ nor $v$. Since all minimal counterexamples have this property, any graph containing any minimal counterexample as an induced subgraph will also have this property, i.e. this holds for all counterexamples, not just minimal ones. Denote by $H_{uv}$ the set of graphs formed by attaching a new vertex $w$ to $H$ in all possible ways while making sure $w$ is not adjacent to $u$ or $v$ (there are only finitely many such attachments). We see that HC-$(\{\overline{K_3}\}\cup H_{uv})$ holds. In fact, we can remove from $H_{uv}$ all graphs containing $\overline{K_3}$ and the result will still hold, since of course no $\overline{K_3}$-free graph could contain a graph that itself contains $\overline{K_3}$. Let $H_{uv}^-$ denote this ``reduced'' set of graphs containing no $\overline{K_3}$.

Suppose that some $H'\in H_{uv}^-$ has a dominating edge $u'v'$. Applying the same logic as before, we see that HC-$(\{\overline{K_3}\}\cup (H_{uv}^-\setminus H')\cup (H')_{u'v'}^-)$ holds. We can repeat this process for as long as we like, at least until we run into graphs with no more dominating edges.

If we impose some additional restriction on counterexamples, say that they satisfy some property $P$, and modify this process to also remove from $H_{uv}$ all graphs satisfying $P$, in some cases we can completely run out of graphs, concluding that all counterexamples HC-$\{\overline{K_3}\}$ satisfy $P$. This is possible when $P$ is \textit{monotone}, that is, if $H$ satisfies $P$ and $G$ has induced subgraph $H$ then $G$ also satisfies $P$. Consider then Algorithm~\ref{alg:dominating}.

\begin{algorithm}[htb!]
    \KwIn{Monotone property $P$, graph $H$ that is known to be an induced subgraph of any counterexample to HC-$\{\overline{K_3}\}$.}
    Let $A\gets\{H\}$\;
    \While{$A$ is nonempty}{
        Let $H'\in A$\;
        Let $A\gets A\setminus\{H'\}$\;
        \If{$H'$ has no dominating edges}{
            \Return{\textsc{failure}}\;
        }\Else{
            Let $uv$ be a dominating edge of $H'$\;
            \For{$N\in 2^{V(H')\setminus \{u,v\}}$}{
                Let $G$ be $H'$ plus a new vertex $w$ such that the neighbors of $w$ are $N$\;
                \If{$G$ has no stable set of size 3, $G$ does not satisfy $P$, and $G$ is $A$-free}{
                    Let $A\gets A\cup\{G\}$\;
                }
            }
        }
    }
    \Return{\textsc{success}}\;
    \caption{Prove that all counterexamples to HC-$\{\overline{K_3}\}$ satisfy monotone property $P$.}
    \label{alg:dominating}
\end{algorithm}

We now prove the correctness of the algorithm.

\begin{lemma}\label{lem:dominatingalg}
If Algorithm~\ref{alg:dominating} returns \textsc{success}, then all counterexamples to HC-$\{\overline{K_3}\}$ satisfy property $P$.
\end{lemma}
\begin{proof}
We claim that, inductively, at the start of each iteration of the main \textbf{while} loop, $A$ is always set of graphs such that all counterexamples to HC-$\{\overline{K_3}\}$ either satisfy $P$ or contains one of the graphs in $A$ as an induced subgraph. The base case, when $A=\{H\}$, is clear.

Suppose $H'\in A$ has a dominating edge $uv$. Then if a counterexample to HC-$\{\overline{K_3}\}$ contained $H'$ as induced subgraph, it must also have a vertex $w$ that is nonadjacent to both $u$ and $v$. Thus if a counterexample contained $H'$, it must actually contain one of the graphs $G$ obtained by attaching to $H$ in all possible ways while being nonadjacent to $u$ and $v$. If $G$ has a stable set of size 3 then it is obviously not a counterexample to HC-$\{\overline{K_3}\}$ (and any graph containing $G$ is also not a counterexample to HC-$\{\overline{K_3}\}$). Additionally, if $G$ satisfies $P$, then since $P$ is monotone, any counterexample containing $G$ will also satisfy $P$. Finally, note that the set of $A$-free graphs is the same as the set of $(A\cup \{G\})$-free graphs if $G$ contains one of the graphs in $A$ as an induced subgraph. Thus if a counterexample to HC-$\{\overline{K_3}\}$ contained $H'$ as induced subgraph, either it satisfies $P$ or it contains as an induced subgraph one of the graphs contained in $A$ at the end of the inner \textbf{for} loop (which is the same as $A$ at the beginning of the next iteration of the \textbf{while} loop). This completes the inductive step.

The condition for the \textbf{while} loop to terminate and the algorithm to return \textsc{success} is that $A$ is empty at the start of the \textbf{while} loop. If this is the case, that means that all counterexamples to HC-$\{\overline{K_3}\}$ satisfy $P$ by the claim just proven.
\end{proof}

Of interest in this paper, the property of containing a particular induced subgraph $H'$ is monotone, so in some cases this can be used to prove that all counterexamples to HC-$\{\overline{K_3}\}$ have an induced $H'$, i.e. HC-$\{\overline{K_3}, H'\}$ holds, so long as the graphs $H$ and $H'$ are chosen carefully and we get lucky. For an example, see the proof of case $I_{21}$ in Section~\ref{sec:data} (and see Figure~\ref{fig:intermediate} for a picture of the complement of this graph). However, it is quite likely that one will run into a graph with no dominating edges, no stable set of size three, and with no induced $H'$, in which case the algorithm will return \textsc{failure}. We can often deal with these cases using Lemmas~\ref{lem:fourclique} and~\ref{lem:conings}, by a process described in the next subsection.

\subsection{Getting Unstuck: Four-Clique Covers}

\label{sec:fourclique}

\subsubsection{Assuming No Vertex Dominates \texorpdfstring{$H$}{H}}

The restriction that a graph has no stable set of size three means that if two vertices have a common nonneighbor, they must be adjacent. Thus if $H$ is an induced subgraph of $G$ and no vertex in $G\setminus H$ is adjacent to all vertices in $H$, then the vertices of $G\setminus H$ partition into cliques based on the intersection of their neighborhood with $H$. Additionally, to avoid stable sets of size 3 and monotone property $P$, many pairs of these cliques must be \textit{complete} to each other, i.e. there must be no nonedges between vertices in those two cliques. In many cases, one can find four cliques $Q_1,Q_2,Q_3,Q_4$ that cover $G$ such that
\[ \sum_{i=1}^4 \abs{Q_i\cap V(H)}\ge \abs{H}+2, \]
so by Pigeonhole Principle, $\omega(G)\ge \ceil{\frac{\abs{G}+2}{4}}$. Note that
\[ \ceil*{\frac{\abs{G}+2}{4}}\ge \begin{cases}
\ceil{\frac{\abs{G}}{4}} & \text{if $\abs{G}$ is even,} \\
\ceil{\frac{\abs{G}+3}{4}} & \text{if $\abs{G}$ is odd,}
\end{cases} \]
so by Lemma~\ref{lem:fourclique}, $G$ is not a counterexample to HC-$\{\overline{K_3}\}$. If this is the case, we conclude that all counterexamples with an induced $H$ either satisfy property $P$ or have a vertex adjacent to every vertex in $H$. The graph obtained by adding to $H$ a vertex adjacent to all vertices in $H$ is known as the \textit{coning} of $H$, and we denote it in this paper by $H\vee K_1$ (for we will generalize this construction in a later subsection). In other words, all counterexamples with an induced $H$ either satisfy property $P$ or have an induced $H\vee K_1$.\footnote{This is the main idea in Bosse's proof of HC-$\{\overline{K_3},W_5\}$, with $H=C_5$ and $P$ omitted \cite{bosse}, although that proof used some additional cleverness beyond the algorithm we describe.}

The usefulness of this is that $H\vee K_1$ has a dominating vertex, and therefore dominating edges, for all $H$. Hence it can be used to ``unstick'' Algorithm~\ref{alg:dominating} when one encounters a graph with no dominating edges, i.e. when the algorithm would otherwise return \textsc{failure}. More importantly, this check can be done completely automatically by Algorithm~\ref{alg:fourclique}.

\begin{algorithm}[htb!]
    \KwIn{Monotone property $P$, graph $H$.}
    Let $A\gets\varnothing$\;
    \For{$N\in 2^{V(H)}\setminus\{V(H)\}$}{
        Let $G$ be $H$ plus a new vertex $w$ such that the neighbors of $w$ are $N$\;
        \If{$G$ has no stable set of size 3 and $G$ does not satisfy $P$}{
            Let $A\gets A\cup\{N\}$\;
        }
    }
    Let $B\gets\varnothing$\;
    \For{$(N_1,N_2)\in\binom{A}{2}$}{
        Let $G$ be $H$ plus new vertices $u$ and $v$ such that the neighbors of $u$ are $N_1$ and neighbors of $v$ are $N_2$\;
        \If{$G$ has a stable set of size 3 or $G$ satisfies $P$}{
            Let $B\gets B\cup\{\{N_1,N_2\}\}$\;
        }
    }
    Let $G$ be $H$ plus vertices $\{u_i\}_{i=1}^{\abs{A}}$ such that the neighbors of each $u_i$ are $N_i$ (where $A=\{N_1,\dots,N_{\abs{A}}\}$), and $u_i$ and $u_j$ are adjacent iff $\{N_i,N_j\}\in B$\;
    \If{there are four cliques $Q_1,Q_2,Q_3,Q_4$ covering $G$ such that $\sum_{i=1}^4 \abs{Q_i\cap V(H)}\ge \abs{H}+2$}{
        \Return{\textsc{success}}\;
    }\Else{
        \Return{\textsc{failure}}\;
    }
    \caption{Prove that all counterexamples to HC-$\{\overline{K_3}\}$ that have an induced $H$ either satisfy monotone property $P$ or have an induced $H\vee K_1$.}
    \label{alg:fourclique}
\end{algorithm}

We now prove the correctness of this algorithm.

\begin{lemma}\label{lem:fourcliquealg}
If Algorithm~\ref{alg:fourclique} returns \textsc{success}, then all counterexamples to HC-$\{\overline{K_3}\}$ with an induced $H$ either satisfy property $P$ or have an induced $H\vee K_1$.
\end{lemma}
\begin{proof}
Suppose for contradiction Algorithm~\ref{alg:fourclique} returned \textsc{success} but there is a counterexample $G'$ to HC-$\{\overline{K_3}\}$ with an induced $H$ but no induced $H\vee K_1$ that does not satisfy property $P$.

Since there is no vertex in $G'$ adjacent to all vertices in the induced $H$, the vertices of $G'\setminus H$ partition into at most $2^{\abs{H}}-1$ cliques based on their neighbors in $H$. Specifically, if two vertices $u,v$ have identical neighbors in $H$, then by assumption that $G'$ has no indcued $H\vee K_1$, $u$ and $v$ have a common nonneighbor $w\in H$ and must be therefore be adjacent to avoid $\{u,v,w\}$ being a stable set of size 3. Some of these potential cliques may be forced to be empty; for instance if $u,v$ are not adjacent in $H$, all vertices in $G'\setminus H$ must be adjacent to at least one of $u$ and $v$. It is clear that, after the first \textbf{for} loop, $A$ contains the set of possible neighborhoods in $H$ of vertices in $G'\setminus H$, since the algorithm filters out all neighborhoods such that if there was a vertex with that neighborhood, $G'$ would either have a stable set of size 3 or satisfy property $P$ (here using the fact that $P$ is assumed to be monotone). Note that not all possible neighborhoods are achieved in $G'$, but all vertices in $G'\setminus H$ have one of the neighborhoods in $A$. We identify the cliques in the clique partition of $G'\setminus H$ with the elements of $A$.

Next, suppose $u\in V(G)$ is in clique $N_1$ and $v\in V(G)$ is in clique $N_2$, for some $N_1\ne N_2$. If the graph formed by adding to $H$ the vertices $u$ and $v$ with neighborhoods $N_1$ and $N_2$, respectively, has either a stable set of size 3 or satisfies $P$, then $u$ and $v$ must be adjacent in $G'$. Since this holds for all vertices in $N_1$ and $N_2$, in fact $N_1$ and $N_2$ are complete to each other in $G'$. Then it is clear that, after the second \textbf{for} loop, if $\{N_1,N_2\}\in B$, then $N_1$ and $N_2$ are complete in $G'$ (though it is possible other pairs of cliques in $G'\setminus H$ could be complete to each other but are not captured by this \textbf{for} loop).

Since the algorithm returned \textsc{success}, there is a covering of the graph $G$ formed after the second \textbf{for} loop by four cliques $Q_1,Q_2,Q_3,Q_4$ such that $\sum_{i=1}^4 \abs{Q_i\cap V(H)}\ge \abs{H}+2$. Then there is an analogous clique cover of $G'$ by four cliques $Q_1',Q_2',Q_3',Q_4'$ such that $\sum_{i=1}^4 \abs{Q_i'\cap V(H)}\ge \abs{H}+2$. In particular,
\[ Q_i'=(Q_i\cap V(H))\cup \bigcup_{u_j\in Q_i} (\text{clique in $G'$ corresponding to }N_j) \]
for $i=1,2,3,4$ is such a clique partition. The inequality $\sum_{i=1}^4 \abs{Q_i'\cap V(H)}\ge \abs{H}+2$ is clear, so we just need to check that the $Q_i'$ are indeed cliques. For two $u,v\in Q_i'$ with $u\ne v$, there are four cases to check:
\begin{itemize}
    \item If $u,v\in V(H)$, then since $Q_i$ is a clique in $G$, $uv$ is an edge of $H$, so an edge in $G'$.
    \item If $u\in V(H)$ and $v\in (\text{clique corresponding to }N_j)$ for some $j$ (or symmetrically $v\in V(H)$ and $u\in \text{clique corresponding to }N_j$), then since $Q_i$ is a clique in $G$, $uu_j$ is an edge of $G$, which means $u\in N_j$ and $uv$ is an edge in $G'$ by definition of $N_j$.
    \item If $u,v\in (\text{clique corresponding to }N_j)$ for some $j$, then obviously $uv$ is an edge in $G'$.
    \item If $u\in (\text{clique corresponding to }N_j)$ and $v\in (\text{clique corresponding to }N_k)$ for some $j\ne k$, then since $Q_i$ is a clique in $G$, $u_ju_k$ is an edge of $G$, so $\{N_j,N_k\}\in B$ and the cliques corresponding to $N_j$ and $N_k$ are complete to each other in $G'$, i.e. $uv$ is an edge in $G'$.
\end{itemize}
This implies by Pigeonhole Principle that $\omega(G')\ge \ceil{\frac{\abs{G'}+2}{4}}$, but by Lemma~\ref{lem:fourclique} this contradicts that $G'$ is a counterexample to HC-$\{\overline{K_3}\}$.
\end{proof}

\subsubsection{Weaker or Stronger Assumptions}

Sometimes we can get different results than Algorithm~\ref{alg:fourclique} with different assumptions than there being no vertex adjacent to all vertices in an induced $H$. For instance, if we allow $G$ to have vertices complete to the induced $H$, but assume that they form a clique, then all vertices in $G$ still partition into cliques based on the intersection of their neighborhood with $H$, only now this neighborhood is allowed to be all of $H$. This assumption implies that $c(G, H)$ is $\overline{K_2}$-free, where $c(G, H)$ is the subgraph of $G\setminus H$ consisting of vertices complete to $H$.

We can make an even weaker assumption on $G$ if we enforce only that $c(G, H)$ has clique cover number equal to at most some constant $k$. This is equivalent to the assumption that $\overline{c(G, H)}$ is $k$-colorable, since a cover of $c(G, H)$ by $k$ cliques is a cover of $\overline{c(G, H)}$ by $k$ independent sets, i.e. color classes. By Lemma~\ref{lem:conings}, $c(G, H)$ having clique cover number at most $k$ is implied by $c(G, H)$ being $F_k$-free, where $F_0=\overline{K_1}\cong K_1$, $F_1=\overline{K_2}$, $F_2=\overline{P_4}\cong P_4$, and $F_3=\overline{T_1}$.\footnote{One could also take $F_3=\overline{T_2}$ or $\overline{T_3}$. This leads to slight variations in the algorithm and could potentially lead to different results. We have not investigated these alternatives.} Denote by $G\vee H$ (read ``$G$ join $H$'') the graph formed by adding to $G\cup H$ all edges between $G$ and $H$. In other words, $G\vee H=\overline{\overline{G}\sqcup \overline{H}}$, with $\sqcup$ the disjoint union of graphs. Then consider Algorithm~\ref{alg:fourcliqueadvanced}.

\begin{algorithm}[htb!]
    \KwIn{Monotone property $P$, graph $H$, integer $k\in\{0,1,2,3\}$.}
    Let $A\gets\varnothing$\;
    \For{$N\in 2^{V(H)}\setminus\{V(H)\}$}{
        Let $G$ be $H$ plus a new vertex $w$ such that the neighbors of $w$ are $N$\;
        \If{$G$ has no stable set of size 3 and $G$ does not satisfy $P$}{
            Let $A\gets A\cup\{N\}$\;
        }
    }
    Let $B\gets\varnothing$\;
    \For{$(N_1,N_2)\in\binom{A}{2}$}{
        Let $G$ be $H$ plus new vertices $u$ and $v$ such that the neighbors of $u$ are $N_1$ and neighbors of $v$ are $N_2$\;
        \If{$G$ has a stable set of size 3 or $G$ satisfies $P$}{
            Let $B\gets B\cup\{\{N_1,N_2\}\}$\;
        }
    }
    Let $C\gets\varnothing$\;
    \For{$N\in A$}{
        Let $G$ be $H$ plus new vertices $u$ and $w$ such that the neighbors of $u$ are $N$ and neighbors of $w$ are $V(H)$\;
        \If{$G$ has a stable set of size 3 or $G$ satisfies $P$}{
            Let $C\gets C\cup\{\{N\}\}$\;
        }
    }
    Let $G$ be $H$ plus vertices $\{u_i\}_{i=1}^{\abs{A}}\cup\{w_i\}_{i=1}^k$ such that the neighbors of each $u_i$ are $N_i$ (where $A=\{N_1,\dots,N_{\abs{A}}\}$), $u_i$ and $u_j$ are adjacent iff $\{N_i,N_j\}\in B$, the neighbors of each $w_i$ are $V(H)$, none of the $w_i$ and $w_j$ are adjacent, and $w_i$ and $u_j$ are adjacent iff $N_j\in C$\;
    \If{there are four cliques $Q_1,Q_2,Q_3,Q_4$ covering $G$ such that $\sum_{i=1}^4 \abs{Q_i\cap V(H)}\ge \abs{H}+2$}{
        \Return{\textsc{success}}\;
    }\Else{
        \Return{\textsc{failure}}\;
    }
    \caption{Prove that all counterexamples to HC-$\{\overline{K_3}\}$ that have an induced $H$ either satisfy monotone property $P$ or have an induced $H\vee F_k$, for chosen $k$.}
    \label{alg:fourcliqueadvanced}
\end{algorithm}

The correctness of this algorithm follows from essentially the same logic as Algorithm~\ref{alg:fourclique}, so we skip some details of the proof.

\begin{lemma}\label{lem:fourcliqueadvancedalg}
If Algorithm~\ref{alg:fourcliqueadvanced} returns \textsc{success}, then all counterexamples to HC-$\{\overline{K_3}\}$ with an induced $H$ either satisfy property $P$ or have an induced $H\vee F_k$.
\end{lemma}
\begin{proof}
The crux of the argument, as in the previous proof, is that if the $G$ constructed by Algorithm~\ref{alg:fourcliqueadvanced} has a cover by cliques $\{Q_1, Q_2, Q_3, Q_4\}$ with $\sum_{i=1}^4\abs{Q_i\cap V(H)}\ge \abs{H}+2$, then any potential counterexample $G'$ also has a cover by cliques $\{Q_1', Q_2', Q_3', Q_4'\}$ with $\sum_{i=1}^4\abs{Q_i'\cap V(H)}\ge \abs{H}+2$, and therefore has clique number at least $\frac{\abs{G'}+2}{4}$. Now $c(G', H)$ is assumed to have clique cover number at most $k$, so it is covered by $k$ cliques $R_1,\dots, R_k$. The cliques of $G'$ are given by
\[ Q_i'=(Q_i\cap V(H))\cup \bigcup_{u_j\in Q_i} (\text{clique in $G'$ corresponding to }N_j)\cup \bigcup_{w_j\in Q_i} R_j. \]
We must check that these are cliques. For $u, v\in Q_i'$ with $u\ne v$, the new cases not in the previous proof are:
\begin{itemize}
    \item If $u\in V(H)$ and $v\in R_j$ for some $j$, then obviously $uv$ is an edge of $G$.
    \item If $u\in \text{clique corresponding to }N_j$ and $v\in R_\ell$ for some $j,\ell$, then $N_j\in C$, so if $uv$ was a nonedge of $G'$, $G'$ would either have a stable set of size 3 or would satisfy $P$.
    \item If $u,v\in R_j$ for some $j$, then $uv$ is an edge of $G$ by assumption that $R_j$ was a clique.
\end{itemize}
It is not possible for $u\in R_j$ and $v\in R_\ell$ for some $j\ne \ell$ to both be in the same $Q_i'$ since the $w_i\in V(G)$ form a stable set.

Therefore if $c(G', H)$ has clique cover number at most $k$, $G'$ is not a counterexample to HC-$\{\overline{K_3}\}$. So by Lemma~\ref{lem:conings}, $c(G', H)$ has an induced $F_k$, so $G'$ has an induced $H\vee F_k$.
\end{proof}

Recall that the purpose of Algorithms~\ref{alg:fourclique} and~\ref{alg:fourcliqueadvanced} is to ``unstick'' Algorithm~\ref{alg:dominating} by introducing new dominating edges. However, the results of these algorithms in general are much stronger than what is necessary to unstick Algorithm~\ref{alg:dominating}. Define $\dom(H)$ to be the set of graphs formed from $H$ by adding one vertex, such that the resulting graph has at least one dominating edge. Then weakest possible variation of this idea would be Algorithm~\ref{alg:fourclique} but replacing $P$ with the logical-or of $P$ and the negation of being $\dom(H)$-free. Since $H\vee K_1\in \dom(H)$ for any $H$, if this modified algorithm returns \textsc{success} then all counterexamples to HC-$\{\overline{K_3}\}$ that have an induced $H$ either satisfy $P$ or have one of the members of $\dom(H)$ as an induced subgraph.

\subsection{Combined Algorithm}

The algorithms in Section~\ref{sec:fourclique} can be easily incorporated into Algorithm~\ref{alg:dominating} to form Algorithm~\ref{alg:full}.

\begin{algorithm}[htb!]
    \KwIn{Monotone property $P$, graph $H$ that is known to be an induced subgraph of any counterexample to HC-$\{\overline{K_3}\}$, boolean $E$.}
    \If{$H$ satisfies $P$}{
        \Return{\textsc{success}}\;
    }
    Let $A\gets\{H\}$\;
    \While{$A$ is nonempty}{
        Let $H'\in A$\;
        Let $A\gets A\setminus\{H'\}$\;
        \If{$H'$ has a dominating edge}{
            Let $uv$ be a dominating edge of $H'$\;
            \For{$N\in 2^{V(H')\setminus \{u,v\}}$}{
                Let $G$ be $H'$ plus a new vertex $w$ such that the neighbors of $w$ are $N$\;
                \If{$G$ has no stable set of size 3, $G$ does not satisfy $P$, and $G$ is $A$-free}{
                    Let $A\gets A\cup\{G\}$\;
                }
            }
        }\Else{
            \For{$k\in\{3,2,1,0\}$ (in the order written)}{
                \If{Algorithm~\ref{alg:fourcliqueadvanced} returns \textsc{success} when given input $(P, H', k)$}{
                    \If{$H'\vee F_k$ has no stable set of size 3, $H'\vee F_k$ does not satisfy $P$, and $H'\vee F_k$ is $A$-free}{
                        Let $A\gets A\cup\{H'\vee F_k\}$\;
                    }
                    Break out of the inner \textbf{for} loop\;
                }
            }
            \If{the previous \textbf{for} loop was not broken out of}{
                \If{Algorithm~\ref{alg:fourclique} returns \textsc{success} when given input $(P\vee\neg\text{``$\dom(H')$-free''}, H')$}{
                    \For{$H''\in \dom(H')$}{
                        \If{$H''$ has no stable set of size 3, $H''$ does not satisfy $P$, and $H''$ is $A$-free}{
                            Let $A\gets A\cup\{H''\}$\;
                        }
                    }
                }\Else{
                    \If{$H'$ is not a counterexample to HC-$\{\overline{K_3}\}$}{
                        \For{$N\in 2^{V(H')}$}{
                            Let $G$ be $H'$ plus a new vertex $w$ such that the neighbors of $w$ are $N$\;
                            \If{$G$ has no stable set of size 3, $G$ does not satisfy $P$, and $G$ is $A$-free}{
                                Let $A\gets A\cup\{G\}$\;
                            }
                        }
                    }\Else{
                        \Return{\textsc{failure}}\;
                    }
                }
            }
        }
    }
    \Return{\textsc{success}}\;
    \caption{Prove that all counterexamples to HC-$\{\overline{K_3}\}$ satisfy monotone property $P$.}
    \label{alg:full}
\end{algorithm}

The correctness of this algorithm follows almost immediately from Lemmas~\ref{lem:dominatingalg},~\ref{lem:fourcliquealg},~and \ref{lem:fourcliqueadvancedalg}. There are just two parts of the algorithm that have not been explained. First, we add a check at the beginning of the algorithm, and if $H$ already satisfies $P$ then we immediately return success. Second, the final check performed is that if $H'$ is not a counterexample to HC-$\{\overline{K_3}\}$, then any counterexample to HC-$\{\overline{K_3}\}$ that has an induced $H'$ must actually have an $H'$ plus one vertex, though the edges incident to this vertex have no restriction other than that the resulting graph is $\overline{K_3}$-free and does not satisfy $P$.

In actuality, Algorithm~\ref{alg:full} never had to perform the final check on any input we have given it in the process of discovering Theorem~\ref{thm:main}. Thus we never actually implemented the final check. However, it is not difficult to construct examples of $H'$ such that the second-to-last check (Algorithm~\ref{alg:fourclique} run with modified $P$) returns \textsc{failure}; for instance this will happen if it is given the complement of any triangle-free 5-chromatic graph. The smallest such graphs have 22 vertices \cite{trianglefree5chromatic}. Thus if one was to extend our computations substantially farther, one might want to add this additional check.

This algorithm's behavior is dependent on the method of choosing $H'$ and a dominating edge $uv$ in cases that $A$ has more than one element or $H'$ has more than one dominating edge. There are theoretical reasons to prefer some methods over others, which we describe in the next subsection. We describe some miscellaneous implementation details in Appendix~\ref{app:implementation}.

\subsection{Behavior of the Algorithm}
\label{sec:behavior}

\subsubsection{Ordering}

As we mentioned, the intermediate steps of the algorithm depend on the method employed to choose $H'$ from $A$ and one dominating edge from $H'$ in the event there are multiple possibilities. In this paper we are only interested in deterministic ways of doing this. Let $\Ac$ be one particular implementation of Algorithm~\ref{alg:full}, i.e. one possible method of choosing graphs and edges. More precisely, $\Ac$ consists of a collection of (computable, deterministic) functions: one that chooses an element $H'\in A$, one that chooses a dominating edge $uv\in H'$ if $H'$ has at least one such edge, and some that determine the order to consider graphs in the innermost \textbf{for} loops. These functions are allowed to take as input auxiliary memory that may be modified at any point in the algorithm. For instance, $\Ac$ may remember some vertex labeling in graphs, remember the order that graphs were added or removed from $A$, remember the number of iterations of various loops, etc. and use this information to determine which graph or dominating edge to consider next.

Let us write $H\vdash_\Ac P$ (read ``$H$ proves $P$ (under implementation $\Ac$)'') if the $\Ac$ implementation of Algorithm~\ref{alg:full} returns \textsc{success} given input $(H, P)$, and write $H\vdash_\Ac G$ if $H\vdash_\Ac P$ when $P$ is the property of having an induced $G$. There are several highly desirable traits that we would like to be true:
\begin{enumerate}
    \item For any $H$, $H\vdash_\Ac H$, i.e. $\vdash_\Ac$ is reflexive.
    \item If $H\vdash_\Ac G$ and $G\vdash_\Ac P$ then $H\vdash_\Ac P$, i.e. $\vdash_\Ac$ is transitive.
    \item If $H$ is an induced subgraph of $H'$ and $H\vdash_\Ac P$ then $H'\vdash_\Ac P$, i.e. $\vdash_\Ac$ is monotone decreasing on the left with respect to taking induced subgraphs.
    \item If $G$ is an induced subgraph of $G'$ and $H\vdash_\Ac G'$ then $H\vdash_\Ac G$, i.e. $\vdash_\Ac$ is monotone increasing on the right with respect to taking induced subgraphs. More generally, if $P$ is a stricter property than $P'$ in the sense that any graph satisfying $P'$ also satisfies $P$, and $H\vdash_\Ac P'$, then $H\vdash_\Ac P$.
\end{enumerate}
This would make $\vdash_\Ac$, when restricted to the case that $P$ is the property of having an induced $G$, a quasi-order on graphs that respects the induced subgraph partial order. The first property is obviously true for any $\Ac$ due to the check at the very beginning of Algorithm~\ref{alg:full}. Additionally, there is a simple condition on $\Ac$ needed to make property (4) hold:

\begin{theorem}\label{thm:fullnice}
If $\Ac$ has the following properties, then if $P$ is a stricter property than $P'$ and $H\vdash_\Ac P'$, then $H\vdash_\Ac P$:
\begin{itemize}
    \item $\Ac$ always chooses the most recently added $H'$ from $A$,
    \item the order that graphs are added to $A$ depends only on the $H'$ chosen in that iteration of the \textbf{while} loop (not on $A$ or $P$ or anything else), and\footnote{Note that this is only a condition on the order that the graphs are considered. Obviously whether or not they are added to $A$ depends on $P$ and $A$, but they must be considered in the same order no matter what.}
    \item the dominating edge chosen when there are multiple possibilities depends on $H'$ only as well.
\end{itemize}
\end{theorem}
\begin{proof}
Suppose $H\vdash_\Ac P'$, and consider the set of graphs $S'$ that are ever a member of $A$ when the algorithm is run with input $(H, P')$. Since $H\vdash_\Ac P'$, this set is finite. Construct a directed acyclic graph $D'$ with vertex set $S'$ and an edge from $v$ and $u$ if $v$ is an induced subgraph of $u$. Then the order that graphs are removed from $A$ corresponds to a depth-first search on $D'$, by virtue of always choosing the latest-added graph to $A$ in each iteration.

If $P$ is a stricter property than $P'$, then the corresponding set $S$ when the algorithm is run with input $(H, P)$ is a subset of $S'$, and the corresponding directed acyclic graph $D$ is a subgraph of $D'$. To see this, consider the sequences $\Sc_1=\{A_1,\dots,A_k\}$ and $\Sc_2=\{A_1',\dots,A_\ell'\}$ consisting of the set $A$ snapshotted at the beginning of each iteration of the main \textbf{while} loop when the algorithm is run with input $(H, P)$ and $(H, P')$, respectively.

We will show there is a injection $f:\Sc_1\to \Sc_2$ such that $A\subset f(A)$ for all $A\in\Sc_1$. More strongly, the relative order that graphs were added to $A$ is the same as the relative order they were added in $f(A)$. This injection is constructed as follows: $f(A_i)$ is the $A_j'$ such that if $H_i\in A_i$ is the graph chosen by $\Ac$ at the beginning of the \textbf{while} loop at that iteration (i.e. the most recently added graph to $A$ at that point in time), then $H_i$ is also the graph chosen in $A_j'$. To see that this is well-defined and has the desired properties, we use induction. This is obviously true for $A_1=\{H\}=A_1'=f(A_1)$.

To see that $f(A_{i+1})$ exists and has the desired properties given $f(A_i)$ does, either $H_{i+1}$ is in $A_i$ or not. If it is, then $H_{i+1}$ is also in $A_j'=f(A_i)$. On some iteration $j_2>j$ of the algorithm run on $(H, P')$, $H_{i+1}$ must be the most recently added graph in $A_{j_2}$. This is the chosen value of $f(A_{i+1})$. The fact that the rest of $A_{i+1}$ is in $f(A_{i+1})$ with the same relative order follows from inductive hypothesis because the rest of $A_{i+1}$ is older than $H_{i+1}'$ and also appears in $A_i$ in the same relative order. Thus the rest of $A_{i+1}$ appears in $f(A_i)$ and, since the rest is older than $H_{i+1}'$, is untouched until after iteration $j_2$.

If $H_{i+1}$ is not in $A_i$, then it was added to $A$ in iteration $i$ under input $(H, P)$, possibly alongside some other graphs. Therefore if $A_j'=f(A_i)$, $H_{i+1}$ was added to $A$ in iteration $j$ under input $(H, P')$, along with the same other graphs as in input $(H, P)$ in the same order, plus possibly some additional graphs interspersed with these (since $P$ is stricter than $P'$). Thus at some iteration $j_2>j$, $H_{i+1}$ will be the most recently added graph under input $(H, P')$, and the rest of the logic follows similarly to the case in the previous paragraph. This completes the inductive step.

Thus we can conclude $S\subseteq S'$ and, more weakly, $H\vdash_\Ac P$.
\end{proof}

Unfortunately, we are unable to find a version of Algorithm~\ref{alg:full} that makes all of the remaining properties true. However, restricted to only the constituent parts of Algorithm~\ref{alg:full} (namely Algorithms~\ref{alg:dominating} and~\ref{alg:fourcliqueadvanced}), we can come up with results close to the above properties. We write $H\vdash_e P$ if either $H$ satisfies $P$ or Algorithm~\ref{alg:dominating} returns \textsc{success} given input $(H, P)$ and the implementation described in the following two paragraphs. Similarly we write $H\vdash_c P$ if either $H$ satisfies $P$ or Algorithm~\ref{alg:fourclique} returns \textsc{success} given input $(H, P)$. Note that the output of Algorithm~\ref{alg:fourclique} does not depend on implementation details.

There is a method of choosing dominating edges and elements of $A$ that makes $\vdash_e$ have these properties. Namely, we first label the vertices of $H$ with integers 1 through $\abs{V(H)}$. Whenever a vertex is added to a graph, it gets the smallest integer label not yet used in that graph. For $H'\vee F_k$, the labels of the vertices in the $F_k$ part are the next $\abs{V(F_k)}$ integers, and relative ordering of vertices in the $F_k$ part is always the same (though this is not strictly necessary). Edges $uv$ are identified with the tuple $(\max(u,v), \min(u,v))$, where vertices are identified by their label. The chosen dominating edge is the minimum edge under the lexicographical ordering with this identification. In essence, the dominating edge chosen is the one whose younger endpoint is as old as possible, breaking ties by looking at the other endpoint.

Next, whenever multiple graphs are to be added to $A$, notice that they are all just $H'$ plus one vertex, possibly with some restrictions on the neighbors of the new vertex. We add these to $A$ in the order such that if $N_1$ and $N_2$ are two different sets of the labels of the neighbors of the newly-added vertex, the graph corresponding to $N_1$ is added before $N_2$ iff, after sorting $N_1$ and $N_2$ and writing them as tuples, $N_1$ comes lexicographically after $N_2$. If $N_1$ is a prefix of $N_2$, it comes after $N_2$; otherwise the one that comes second is the one with the smallest element in the symmetric difference of $N_1$ and $N_2$.

\begin{theorem}\label{thm:nicedominating}
Under the method of choosing dominating edges and graphs described in the previous paragraphs:
\begin{enumerate}
    \item If $H\vdash_e G$ and $G\vdash_e P$ then $H\vdash_e P$.
    \item If $H$ is an induced subgraph of $H'$ and $H\vdash_e P$ then $H'\vdash_e P$.
    \item If $H\vdash_e P$ then $H\vdash_\Ac P$, where the method for choosing graphs and dominating edges in $\Ac$ is also as in the previous paragraphs.
\end{enumerate}
\end{theorem}
\begin{proof}
First, (3) is obvious since Algorithm~\ref{alg:full} performs exactly the same computations as Algorithm~\ref{alg:dominating} given that implementation $\Ac$ and the fact that $H\vdash_e P$.


For a particular graph $H$, consider the set $S(H)$ of minimal graphs $H'$ containing $H$ so that no dominating edge of $H$ is a dominating edge of $H'$. If $H\vdash_e P$, then for each $H'\in S(H)$, either $H'$ satisfies $P$ or $H'$ has a dominating edge, because if $H'$ did not have a dominating edge and did not satisfy $P$, then Algorithm~\ref{alg:dominating} would construct it from $H$ and then return \textsc{failure}. In particular, because the order for choosing dominating edges considers all dominating edges of $H$ before any edges introduced by adding vertices to $H$, all members $H'$ of $S(H)$ will be constructed by the algorithm unless a proper induced subgraph of $H'$ satisfies $P$. Let $S_P(H)$ be the subset of $S(H)$ that does not satisfy property $P$. 

If $H\vdash_e P$, then for each $H'\in S_P(H)$, for each $H''\in S(H')$, either $H''$ satisfies $P$ or $H''$ has a dominating edge. We continue in this manner, but because $H\vdash_e P$, the algorithm terminates at some point, and there must be a $k$ such that $S_P^k(H)$ is empty, where $S_P^0(H)=\{H\}$ (unless $H$ already satisfies $P$, in which case $S_P^0(H)=\varnothing$) and
\[ S_P^i(H) = \bigcup_{H'\in S_P^{i-1}(H)} S_P(H') \]
for $i\ge 1$.

Now suppose $H$ is an induced subgraph of $I$ and $H\vdash_e P$. Then for any graph $I'\in S(I)$, there is a graph $H'\in S(H)$ so that $H'$ is an induced subgraph of $H'$. This is because in $I'$ none of the dominating edges of $I$ are dominating, which means none of the dominating edges of the $H$ contained in the $I$ are dominating in $I'$, so $I'$ must contain one of the graphs in $S(H)$. Similarly, for any $I'\in S_P(I)$, there is an $H'\in S_P(H)$ such that $H'$ is an induced subgraph of $I'$. Inductively one can then see that for any $i$ and $I'\in S_P^i(I)$, there is an $H'\in S_P^i(H)$ such that $H'$ is an induced subgraph of $I'$. Setting $i=k$, since $S_P^k(H)$ is empty, $S_P^k(I)$ must also be empty, which means $I\vdash_e P$. This is (2).

For (1), let $k_G$ be the minimum $k$ so that $S_P^k(G)=\varnothing$ and $k_H$ the minimum $k$ so that $S_G^k(H)=\varnothing$, where the subscript $G$ means to consider the property of having an induced $G$. Any graph $I\in S_P^{k_H}(H)$ has an induced $G$ and therefore (by the reasoning in the previous paragraph), $S_P^{k_G}(I)=\varnothing$. But it is to prove by induction on $b$ that
\[ S_P^{a+b}(H) = \bigcup_{I\in S_P^a(H)} S_P^b(I), \]
which proves (1) by setting $a=k_H$ and $b=k_G$.

\end{proof}

Similarly, for Algorithm~\ref{alg:fourclique}:

\begin{theorem}\label{thm:nicefourclique}
If $H$ is an induced subgraph of $H'$, $H\vee K_1$ satisfies $P$, and $H\vdash_c P$ then $H'\vdash_c P$.
\end{theorem}
\begin{proof}
Note if $H'$ has $H\vee K_1$ as an induced subgraph, $H'\vdash_c P$ is immediate. If not, then consider first the case that $H'$ is $H$ plus one vertex $v$, and consider the sets $A$ and $B$ constructed for the algorithm given input $(H, P)$ and corresponding sets $A'$ and $B'$ constructed given input $(H', P)$. Then since $H\vee K_1$ satisfies $P$, we find every member of $A'$ is either $N_i$ or $N_i'=N_i\cup \{v\}$ for some $N_i\in A$. Also, if the neighbors of $v$ are $N$, then $N\ne V(H)$, and $N$ is not in $A'$ (only possibly $N\cup \{v\}$) due to the restriction that there are no stable sets of size 3.

Now the following are members of $B'$, assuming the relevant sets are members of $A'$:
\begin{itemize}
    \item $\{N_i, N_i'\}$, since there is some $v\in H$ with $v\not\in N_i$ and $v\not\in N_i'$.
    \item $\{N_i, N_j\}$, $\{N_i, N_j'\}$, and $\{N_i', N_j'\}$ if $\{N_i, N_j\}\in B$, since if $H$ plus two vertices with neighbor sets $N_i$ and $N_j$ has a stable set of size 3 or satisfies $P$, $H'$ plus two vertices with neighbor sets ($N_i$ or $N_i'$) and ($N_j$ or $N_j'$) also has a stable set of size 3 or satisfies $P$, which can be seen just by deleting $v$ from the graph after adding the two new vertices.
\end{itemize}

Thus if $Q_1,Q_2,Q_3,Q_4$ are four cliques covering $G=H\cup\{u_i\}$ with $\sum_{i=1}^4\abs{Q_i\cap V(H)}\ge \abs{H}+2$, then there are $Q_1',Q_2',Q_3',Q_4'$ are four cliques covering $G'=H'\cup\{u_i\}\cup\{u_i'\}$ with $\sum_{i=1}^4\abs{Q_i'\cap V(H')}\ge \abs{H'}+2$. Let the neighbors of $v$ in $H$ be $N$, which we can take without loss of generality to be in $Q_1$. Then set
\[ Q_i'=\{u_i\mid u_i\in Q_i\}\cup \{u_i'\mid u_i\in Q_i\}\cup \begin{cases}\{v, u'\}&\text{if $i=1$} \\ \varnothing & \text{otherwise} \end{cases} \]
and $u_i'$ is the vertex in $G'$ corresponding to set $N_i'\in A'$ and $u'$ the vertex corresponding to $N\cup \{v\}$ where $N$ is the neighbor set of $v$ (unless $N\cup \{v\}\not\in A'$, in which case just drop $u'$ from $Q_1'$). The fact these are cliques follows from $Q_i$ being cliques and each of the pairs of sets noted above being in $B'$, and
\begin{align*}
    \sum_{i=1}^4\abs{Q_i'\cap V(H')} &\ge 1 + \sum_{i=1}^4\abs{Q_i'\cap V(H)} & \text{($Q_1$ contains $v$)} \\
    &\ge 1+\sum_{i=1}^4\abs{Q_i\cap V(H)} \\
    &\ge \abs{V(H)}+3 \\
    &= \abs{V(H')}+2.
\end{align*}

The full theorem statement then follows easily by induction on $\abs{V(H')}$.
\end{proof}

Unfortunately, it does not seem like one can generalize the above theorem to be valid for Algorithm~\ref{alg:fourcliqueadvanced} run with the third input $k>0$. For instance, one can verify that the algorithm returns \textsc{success} on the input $(F_1, F_1\vee F_2, 2)$, but \textsc{failure} on $(F_2, F_2\vee F_1, 1)$. Therefore if one wanted to prove HC-$\{\overline{K_3}, F_1\vee F_2\}$ using Algorithm~\ref{alg:fourcliqueadvanced}, you would succeed if you started from $F_1$ but not $F_2$, despite $F_2$ containing $F_1$. There are additional complications if one wanted to combine Theorems~\ref{thm:nicedominating} and~\ref{thm:nicefourclique} to say something about the behavior of Algorithm~\ref{alg:full}, but this specific failure means that there is probably no implementation or simple modification to Algorithm~\ref{alg:full} that makes all the desired properties hold unless one severely weakened the algorithm by removing the calls to Algorithm~\ref{alg:fourcliqueadvanced} with $k>0$ and the additional checks if Algorithm~\ref{alg:fourcliqueadvanced} fails.

\subsubsection{Limitations}
\label{sec:limitations}

We begin this section with a simple observation:

\begin{proposition}\label{prop:complete}
If HC-$\{\overline{K_3}, K_n\}$ holds for some $n$, the $H\vdash_\Ac K_n$ for all $H$ and $\Ac$.
\end{proposition}
\begin{proof}
By Ramsey's Theorem there are only finitely many $\{\overline{K_3}, K_n\}$-free graphs, so Algorithm~\ref{alg:full} must terminate. If HC-$\{\overline{K_3}, K_n\}$ holds then Algorithm~\ref{alg:full} will not return \textsc{failure}, so it must return \textsc{success}.
\end{proof}

For some $H$ and $G$, we can actually prove the reverse: there is no $\Ac$ for which $H\vdash_\Ac G$. Given a graph $I$, define the \textit{$I$-core} of a graph $H$ to be the subgraph of $H$ induced by all vertices that are in an induced $G$ in $H$. Then:

\begin{theorem}\label{thm:corefailure}
If $H\vdash_\Ac G$ for some $\Ac$, then the $I$-core of $G$ is an induced subgraph of the $I$-core of $H$, where $I$ is either:
\begin{enumerate}
    \item $\overline{C_4}$,
    \item any graph with minimum codegree at least 3, i.e. the complement of a graph with minimum degree at least 3.
\end{enumerate}
\end{theorem}
\begin{proof}
We prove the contrapositive: if the $I$-core of $G$ is not an induced subgraph of the $I$-core of $H$, then not $H\vdash_\Ac G$ for any $\Ac$. In particular, we will prove inductively that at each iteration of the \textbf{while} loop, $A$ contains at least one graph whose $I$-core does not contain the $I$-core of $G$. The base case (in either possibility for $I$) is just the hypothesis of the contrapositive. Also, in either possibility for $I$, if the $I$-core of the chosen $H'\in A$ contains the $I$-core of $G$, then by inductive hypothesis there is another $H''\in A$ whose $I$-core does not contain the $I$-core of $G$, and this $H''$ will still be present in $A$ in the next iteration (unless the algorithm returned \textsc{failure}). So we may assume the chosen $H'\in A$ has an $I$-core that does not contain the $I$-core of $G$.
\begin{enumerate}
    \item If $H'$ has a dominating edge, say $\Ac$ chooses edge $uv$, then one possible neighbor set of the newly added vertex $w$ is $V(H')\setminus\{u,v\}$. Let $H''$ be the graph formed by adding $w$ to $H'$. In this case, $w$ cannot possibly be in an induced $\overline{C_4}$ in $H''$. If it was, say $\{a,b,c,w\}$ induces $\overline{C_4}$ with $a$ adjacent to $b$ and $c$ adjacent to $w$. Since $w$ is not adjacent to $a$ or $b$, we must have $\{a,b\}=\{u,v\}$, but then $uv$ was not dominating since $c$ is also nonadjacent to both $a$ and $b$. Thus the $\overline{C_4}$-core of $H''$ is the same as $H'$. Therefore $G$ is not an induced subgraph of $H'$. Also $H''$ is $\overline{K_3}$-free if $H'$ is, since the only nonneighbors of $w$ are $u$ and $v$, which are adjacent. Therefore either $H''$ is added to $A$ or else there is another $H'''\in A$ that is an induced subgraph of $H''$. In either case, there is a graph whose $\overline{C_4}$-core does not contain the $\overline{C_4}$-core of $G$.
    
    If $H'$ has no dominating edge, then if any of the calls to Algorithm~\ref{alg:fourcliqueadvanced} succeed, one can verify the new graph to be added to $A$ has the same $\overline{C_4}$-core as $H'$, so the inductive step follows in this case as well by similar logic to above.
    
    If none of the calls to Algorithm~\ref{alg:fourcliqueadvanced} succeed, but the call to Algorithm~\ref{alg:fourclique} succeeds, then one of the graphs to be added to $A$ is $H\vee K_1$. By similar logic to above, the inductive step holds here.
    
    Finally, if even the call to Algorithm~\ref{alg:fourclique} fails, but the final check that $H'$ is not a counterexample to HC-$\{\overline{K_3}\}$ succeeds, then one of the graphs to be added to $A$ is $H\vee K_1$. By the same logic, the inductive step holds.
    
    If all checks fail then the algorithm returns \text{failure} and it is not the case that $H\vdash_\Ac G$.
    \item If $H'$ has a dominating edge, say $\Ac$ chooses edge $uv$, then one possible neighbor set of the newly added vertex $w$ is $V(H')\setminus\{u,v\}$. Noting that $w$ has codegree 2 in $H''$, we see that the $I$-core of $H''$ is the same as the $I$-core of $H'$. Then the inductive step follows from the same logic as in case (1).
    
    If $H'$ has no dominating edge, then if any of the calls to Algorithm~\ref{alg:fourcliqueadvanced} succeed, one can verify the new graph to be added to $A$ has the same $\overline{C_4}$-core as $H'$, so the inductive step follows in this case as well by similar logic to above. Some care must be taken in the $k=3$ case, since $F_3$ has a vertex $w$ of codegree 3. However, none of the nonneighbors of this vertex have codegree 3 or greater, so it is not part of any induced subgraph of $H'\vee F_3$ that has minimum codegree 3 (either an induced subgraph of $H'\vee F_3$ contains a nonneighbor of $w$ with codegree less than 3 or $H'\vee F_3$ does not contain all nonneighbors of $w$ and $w$ has codegree less than 3 in the induced subgraph).
    
    The last two cases, if the call to Algorithm~\ref{alg:fourclique} succeeds or if it fails, are identical to case (1).
\end{enumerate}
\end{proof}

This means that it is impossible for the techniques in this paper to improve Theorem~\ref{thm:past} to all graphs on six vertices, since the graph $\overline{K_{3,3}}$ (among others) is equal to its $\overline{C_4}$-core. It also helps explain the general patterns in the graphs $H_i'$, such as the relative lack of induced $\overline{C_4}$'s (compared to induced $\overline{C_5}$'s, say) and abundance of low-codegree vertices (which appear as leaves and isolated vertices in the drawings of the complements in Figure~\ref{fig:main}).

\section{Results}
\label{sec:data}

The graphs used in intermediate steps in proving Theorem~\ref{thm:main} and in the example proofs are shown in Figure~\ref{fig:intermediate}.

\subsection{Some Simple Cases}

First we present a proof of HC-$\{\overline{K_3}, I_{21}\}$ that uses only Algorithm~\ref{alg:dominating}, i.e. only using the fact that minimal counterexamples to HC-$\{\overline{K_3}\}$ have no dominating edges. This will involve several steps starting from the graph $W_5$.

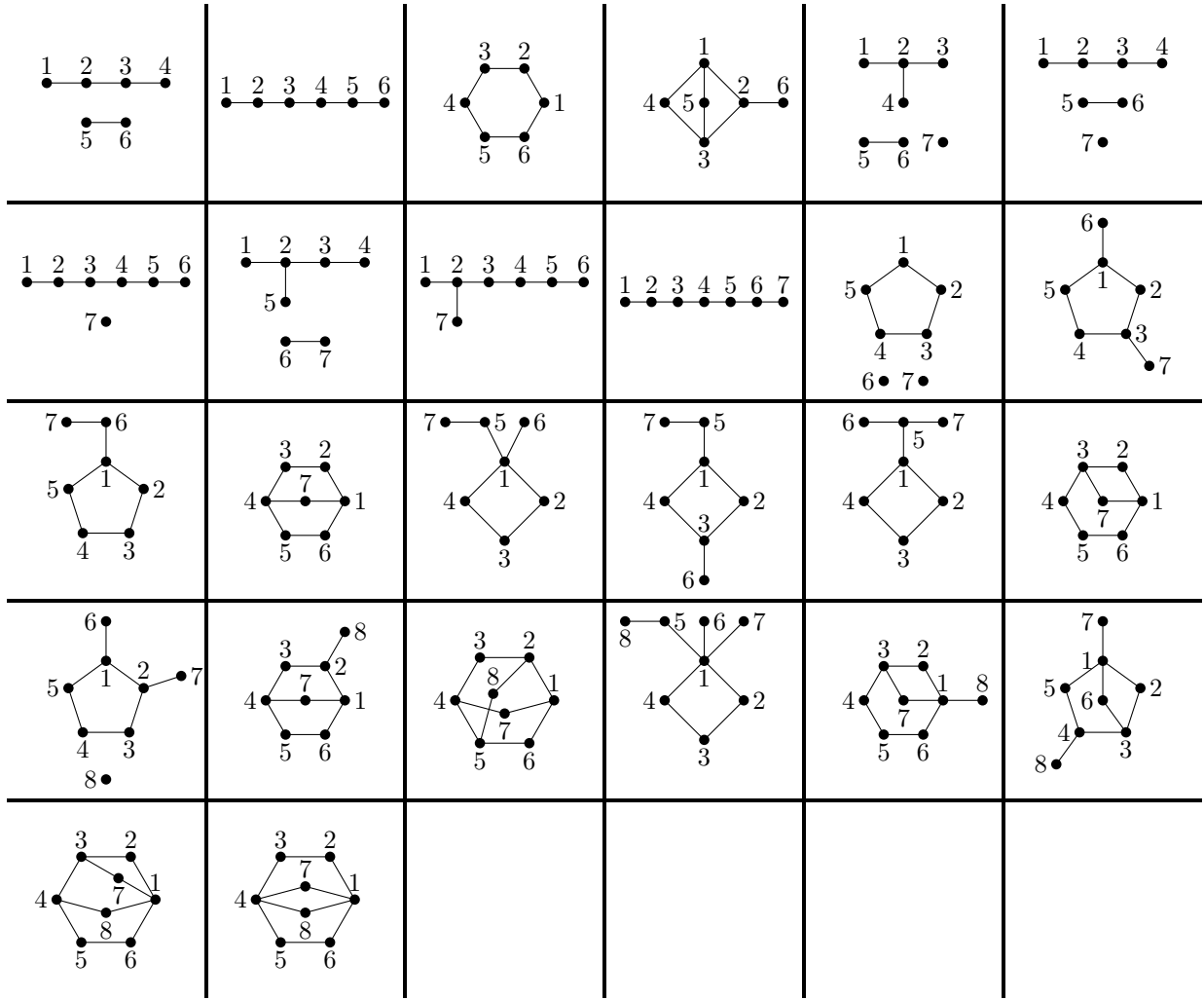
\begin{figure}[htbp!]
    \centering
    \begin{TAB}(e)[0.75em]{c|c|c|c|c|c}{c|c|c|c|c}
    \begin{tikzpicture}[scale=0.54]
    \bound;
    \labvert{1}{-1.5,0.5}{above};
    \labvert{2}{-0.5,0.5}{above};
    \labvert{3}{0.5,0.5}{above};
    \labvert{4}{1.5,0.5}{above};
    \labvert{5}{-0.5,-0.5}{below};
    \labvert{6}{0.5,-0.5}{below};
    \edge{1}{2};
    \edge{2}{3};
    \edge{3}{4};
    \edge{5}{6};
    \end{tikzpicture} &
    \begin{tikzpicture}[scale=0.54]
    \bound;
    \labvert{1}{-2,0}{above};
    \labvert{2}{-1.2,0}{above};
    \labvert{3}{-0.4,0}{above};
    \labvert{4}{0.4,0}{above};
    \labvert{5}{1.2,0}{above};
    \labvert{6}{2,0}{above};
    \edge{1}{2};
    \edge{2}{3};
    \edge{3}{4};
    \edge{4}{5};
    \edge{5}{6};
    \end{tikzpicture} &
    \begin{tikzpicture}[scale=0.54]
    \bound;
    \labvert{1}{1,0}{right};
    \labvert{2}{0.5,0.866}{above};
    \labvert{3}{-0.5,0.866}{above};
    \labvert{4}{-1,0}{left};
    \labvert{5}{-0.5,-0.866}{below};
    \labvert{6}{0.5,-0.866}{below};
    \edge{1}{2};
    \edge{2}{3};
    \edge{3}{4};
    \edge{4}{5};
    \edge{5}{6};
    \edge{6}{1};
    \end{tikzpicture} &
    \begin{tikzpicture}[scale=0.54]
    \bound;
    \labvert{1}{0,1}{above};
    \labvert{2}{1,0}{above};
    \labvert{3}{0,-1}{below};
    \labvert{4}{-1,0}{left};
    \labvert{5}{0,0}{left};
    \labvert{6}{2,0}{above};
    \edge{1}{2};
    \edge{2}{3};
    \edge{3}{4};
    \edge{4}{1};
    \edge{1}{5};
    \edge{3}{5};
    \edge{2}{6};
    \end{tikzpicture} &
    \begin{tikzpicture}[scale=0.54]
    \bound;
    \labvert{1}{-1,1}{above};
    \labvert{2}{0,1}{above};
    \labvert{3}{1,1}{above};
    \labvert{4}{0,0}{left};
    \labvert{5}{-1,-1}{below};
    \labvert{6}{0,-1}{below};
    \labvert{7}{1,-1}{left};
    \edge{1}{2};
    \edge{2}{3};
    \edge{2}{4};
    \edge{5}{6};
    \end{tikzpicture} &
    \begin{tikzpicture}[scale=0.54]
    \bound;
    \labvert{1}{-1.5,1}{above};
    \labvert{2}{-0.5,1}{above};
    \labvert{3}{0.5,1}{above};
    \labvert{4}{1.5,1}{above};
    \labvert{5}{-0.5,0}{left};
    \labvert{6}{0.5,0}{right};
    \labvert{7}{0,-1}{left};
    \edge{1}{2};
    \edge{2}{3};
    \edge{3}{4};
    \edge{5}{6};
    \end{tikzpicture} \\
    \begin{tikzpicture}[scale=0.54]
    \bound;
    \labvert{1}{-2,0.5}{above};
    \labvert{2}{-1.2,0.5}{above};
    \labvert{3}{-0.4,0.5}{above};
    \labvert{4}{0.4,0.5}{above};
    \labvert{5}{1.2,0.5}{above};
    \labvert{6}{2,0.5}{above};
    \labvert{7}{0,-0.5}{left};
    \edge{1}{2};
    \edge{2}{3};
    \edge{3}{4};
    \edge{4}{5};
    \edge{5}{6};
    \end{tikzpicture} &
    \begin{tikzpicture}[scale=0.54]
    \bound;
    \labvert{1}{-1.5,1}{above};
    \labvert{2}{-0.5,1}{above};
    \labvert{3}{0.5,1}{above};
    \labvert{4}{1.5,1}{above};
    \labvert{5}{-0.5,0}{left};
    \labvert{6}{-0.5,-1}{below};
    \labvert{7}{0.5,-1}{below};
    \edge{1}{2};
    \edge{2}{3};
    \edge{3}{4};
    \edge{2}{5};
    \edge{6}{7};
    \end{tikzpicture} &
    \begin{tikzpicture}[scale=0.54]
    \bound;
    \labvert{1}{-2,0.5}{above};
    \labvert{2}{-1.2,0.5}{above};
    \labvert{3}{-0.4,0.5}{above};
    \labvert{4}{0.4,0.5}{above};
    \labvert{5}{1.2,0.5}{above};
    \labvert{6}{2,0.5}{above};
    \labvert{7}{-1.2,-0.5}{left};
    \edge{1}{2};
    \edge{2}{3};
    \edge{3}{4};
    \edge{4}{5};
    \edge{5}{6};
    \edge{2}{7};
    \end{tikzpicture} &
    \begin{tikzpicture}[scale=0.54]
    \bound;
    \labvert{1}{-2,0}{above};
    \labvert{2}{-1.333,0}{above};
    \labvert{3}{-0.667,0}{above};
    \labvert{4}{0,0}{above};
    \labvert{5}{0.667,0}{above};
    \labvert{6}{1.333,0}{above};
    \labvert{7}{2,0}{above};
    \edge{1}{2};
    \edge{2}{3};
    \edge{3}{4};
    \edge{4}{5};
    \edge{5}{6};
    \edge{6}{7};
    \end{tikzpicture} &
    \begin{tikzpicture}[scale=0.54]
    \bound;
    \labvert{1}{0,1}{above};
    \labvert{2}{0.951,0.309}{right};
    \labvert{3}{0.588,-0.809}{below};
    \labvert{4}{-0.588,-0.809}{below};
    \labvert{5}{-0.951,0.309}{left};
    \labvert{6}{-0.5,-2}{left};
    \labvert{7}{0.5,-2}{left};
    \edge{1}{2};
    \edge{2}{3};
    \edge{3}{4};
    \edge{4}{5};
    \edge{5}{1};
    \end{tikzpicture} &
    \begin{tikzpicture}[scale=0.54]
    \bound;
    \labvert{1}{0,1}{below};
    \labvert{2}{0.951,0.309}{right};
    \labvert{3}{0.588,-0.809}{right};
    \labvert{4}{-0.588,-0.809}{below};
    \labvert{5}{-0.951,0.309}{left};
    \labvert{6}{0,2}{left};
    \labvert{7}{1.176,-1.618}{right};
    \edge{1}{2};
    \edge{2}{3};
    \edge{3}{4};
    \edge{4}{5};
    \edge{5}{1};
    \edge{1}{6};
    \edge{3}{7};
    \end{tikzpicture} \\
    \begin{tikzpicture}[scale=0.54]
    \bound;
    \labvert{1}{0,1}{below};
    \labvert{2}{0.951,0.309}{right};
    \labvert{3}{0.588,-0.809}{below};
    \labvert{4}{-0.588,-0.809}{below};
    \labvert{5}{-0.951,0.309}{left};
    \labvert{6}{0,2}{right};
    \labvert{7}{-1,2}{left};
    \edge{1}{2};
    \edge{2}{3};
    \edge{3}{4};
    \edge{4}{5};
    \edge{5}{1};
    \edge{1}{6};
    \edge{6}{7};
    \end{tikzpicture} &
    \begin{tikzpicture}[scale=0.54]
    \bound;
    \labvert{1}{1,0}{right};
    \labvert{2}{0.5,0.866}{above};
    \labvert{3}{-0.5,0.866}{above};
    \labvert{4}{-1,0}{left};
    \labvert{5}{-0.5,-0.866}{below};
    \labvert{6}{0.5,-0.866}{below};
    \labvert{7}{0,0}{above};
    \edge{1}{2};
    \edge{2}{3};
    \edge{3}{4};
    \edge{4}{5};
    \edge{5}{6};
    \edge{6}{1};
    \edge{1}{7};
    \edge{4}{7};
    \end{tikzpicture} &
    \begin{tikzpicture}[scale=0.54]
    \bound;
    \labvert{1}{0,1}{below};
    \labvert{2}{1,0}{right};
    \labvert{3}{0,-1}{below};
    \labvert{4}{-1,0}{left};
    \labvert{5}{-0.5,2}{right};
    \labvert{6}{0.5,2}{right};
    \labvert{7}{-1.5,2}{left};
    \edge{1}{2};
    \edge{2}{3};
    \edge{3}{4};
    \edge{4}{1};
    \edge{1}{5};
    \edge{1}{6};
    \edge{5}{7};
    \end{tikzpicture} &
    \begin{tikzpicture}[scale=0.54]
    \bound;
    \labvert{1}{0,1}{below};
    \labvert{2}{1,0}{right};
    \labvert{3}{0,-1}{above};
    \labvert{4}{-1,0}{left};
    \labvert{5}{0,2}{right};
    \labvert{6}{0,-2}{left};
    \labvert{7}{-1,2}{left};
    \edge{1}{2};
    \edge{2}{3};
    \edge{3}{4};
    \edge{4}{1};
    \edge{1}{5};
    \edge{3}{6};
    \edge{5}{7};
    \end{tikzpicture} &
    \begin{tikzpicture}[scale=0.54]
    \bound;
    \labvert{1}{0,1}{below};
    \labvert{2}{1,0}{right};
    \labvert{3}{0,-1}{below};
    \labvert{4}{-1,0}{left};
    \labvert{5}{0,2}{below right};
    \labvert{6}{-1,2}{left};
    \labvert{7}{1,2}{right};
    \edge{1}{2};
    \edge{2}{3};
    \edge{3}{4};
    \edge{4}{1};
    \edge{1}{5};
    \edge{5}{6};
    \edge{5}{7};
    \end{tikzpicture} &
    \begin{tikzpicture}[scale=0.54]
    \bound;
    \labvert{1}{1,0}{right};
    \labvert{2}{0.5,0.866}{above};
    \labvert{3}{-0.5,0.866}{above};
    \labvert{4}{-1,0}{left};
    \labvert{5}{-0.5,-0.866}{below};
    \labvert{6}{0.5,-0.866}{below};
    \labvert{7}{0,0}{below};
    \edge{1}{2};
    \edge{2}{3};
    \edge{3}{4};
    \edge{4}{5};
    \edge{5}{6};
    \edge{6}{1};
    \edge{1}{7};
    \edge{3}{7};
    \end{tikzpicture} \\
    \begin{tikzpicture}[scale=0.54]
    \bound;
    \labvert{1}{0,1}{below};
    \labvert{2}{0.951,0.309}{above};
    \labvert{3}{0.588,-0.809}{below};
    \labvert{4}{-0.588,-0.809}{below};
    \labvert{5}{-0.951,0.309}{left};
    \labvert{6}{0,2}{left};
    \labvert{7}{1.902,0.618}{right};
    \labvert{8}{0,-2}{left};
    \edge{1}{2};
    \edge{2}{3};
    \edge{3}{4};
    \edge{4}{5};
    \edge{5}{1};
    \edge{1}{6};
    \edge{2}{7};
    \end{tikzpicture} &
    \begin{tikzpicture}[scale=0.54]
    \bound;
    \labvert{1}{1,0}{right};
    \labvert{2}{0.5,0.866}{right};
    \labvert{3}{-0.5,0.866}{above};
    \labvert{4}{-1,0}{left};
    \labvert{5}{-0.5,-0.866}{below};
    \labvert{6}{0.5,-0.866}{below};
    \labvert{7}{0,0}{above};
    \labvert{8}{1,1.732}{right};
    \edge{1}{2};
    \edge{2}{3};
    \edge{3}{4};
    \edge{4}{5};
    \edge{5}{6};
    \edge{6}{1};
    \edge{1}{7};
    \edge{4}{7};
    \edge{2}{8};
    \end{tikzpicture} &
    \begin{tikzpicture}[scale=0.54]
    \bound;
    \labvert{1}{1.25,0}{above};
    \labvert{2}{0.625,1.083}{above};
    \labvert{3}{-0.625,1.083}{above};
    \labvert{4}{-1.25,0}{left};
    \labvert{5}{-0.625,-1.083}{below};
    \labvert{6}{0.625,-1.083}{below};
    \labvert{7}{0,-0.333}{below};
    \labvert{8}{-0.289,0.167}{above};
    \edge{1}{2};
    \edge{2}{3};
    \edge{3}{4};
    \edge{4}{5};
    \edge{5}{6};
    \edge{6}{1};
    \edge{1}{7};
    \edge{4}{7};
    \edge{2}{8};
    \edge{5}{8};
    \end{tikzpicture} &
    \begin{tikzpicture}[scale=0.54]
    \bound;
    \labvert{1}{0,1}{below};
    \labvert{2}{1,0}{right};
    \labvert{3}{0,-1}{below};
    \labvert{4}{-1,0}{left};
    \labvert{5}{-1,2}{right};
    \labvert{6}{0,2}{right};
    \labvert{7}{1,2}{right};
    \labvert{8}{-2,2}{below};
    \edge{1}{2};
    \edge{2}{3};
    \edge{3}{4};
    \edge{4}{1};
    \edge{1}{5};
    \edge{1}{6};
    \edge{1}{7};
    \edge{5}{8};
    \end{tikzpicture} &
    \begin{tikzpicture}[scale=0.54]
    \bound;
    \labvert{1}{1,0}{above};
    \labvert{2}{0.5,0.866}{above};
    \labvert{3}{-0.5,0.866}{above};
    \labvert{4}{-1,0}{left};
    \labvert{5}{-0.5,-0.866}{below};
    \labvert{6}{0.5,-0.866}{below};
    \labvert{7}{0,0}{below};
    \labvert{8}{2,0}{above};
    \edge{1}{2};
    \edge{2}{3};
    \edge{3}{4};
    \edge{4}{5};
    \edge{5}{6};
    \edge{6}{1};
    \edge{1}{7};
    \edge{3}{7};
    \edge{1}{8};
    \end{tikzpicture} &
    \begin{tikzpicture}[scale=0.54]
    \bound;
    \labvert{1}{0,1}{left};
    \labvert{2}{0.951,0.309}{right};
    \labvert{3}{0.588,-0.809}{below};
    \labvert{4}{-0.588,-0.809}{left};
    \labvert{5}{-0.951,0.309}{left};
    \labvert{6}{0,0}{left};
    \labvert{7}{0,2}{left};
    \labvert{8}{-1.176,-1.618}{left};
    \edge{1}{2};
    \edge{2}{3};
    \edge{3}{4};
    \edge{4}{5};
    \edge{5}{1};
    \edge{1}{6};
    \edge{3}{6};
    \edge{1}{7};
    \edge{4}{8};
    \end{tikzpicture} \\
    \begin{tikzpicture}[scale=0.54]
    \bound;
    \labvert{1}{1.25,0}{above};
    \labvert{2}{0.625,1.083}{above};
    \labvert{3}{-0.625,1.083}{above};
    \labvert{4}{-1.25,0}{left};
    \labvert{5}{-0.625,-1.083}{below};
    \labvert{6}{0.625,-1.083}{below};
    \labvert{8}{0,-0.333}{below};
    \labvert{7}{0.313,0.541}{below};
    \edge{1}{2};
    \edge{2}{3};
    \edge{3}{4};
    \edge{4}{5};
    \edge{5}{6};
    \edge{6}{1};
    \edge{1}{8};
    \edge{4}{8};
    \edge{1}{7};
    \edge{3}{7};
    \end{tikzpicture} &
    \begin{tikzpicture}[scale=0.54]
    \bound;
    \labvert{1}{1.25,0}{above};
    \labvert{2}{0.625,1.083}{above};
    \labvert{3}{-0.625,1.083}{above};
    \labvert{4}{-1.25,0}{left};
    \labvert{5}{-0.625,-1.083}{below};
    \labvert{6}{0.625,-1.083}{below};
    \labvert{8}{0,-0.333}{below};
    \labvert{7}{0,0.333}{above};
    \edge{1}{2};
    \edge{2}{3};
    \edge{3}{4};
    \edge{4}{5};
    \edge{5}{6};
    \edge{6}{1};
    \edge{1}{8};
    \edge{4}{8};
    \edge{1}{7};
    \edge{4}{7};
    \end{tikzpicture} & & & &
    \end{TAB}
    \caption{The complements of the graphs $I_1, \dots, I_{26}$ when read left-to-right, top-to-bottom. These graphs are used in the proof of Theorem~\ref{thm:main}. As in Figure~\ref{fig:main}, the labelings are used to choose edges consistently in algorithms.}
    \label{fig:intermediate}
\end{figure}

\begin{lemma}
HC-$\{\overline{K_3}, I_3\}$ holds.
\end{lemma}
\begin{proof}
By Theorem~\ref{thm:past}, any minimal counterexample $G$ to HC-$\{\overline{K_3}\}$ has an induced $W_5$, say induced by vertices 1 through 6 with 1-2-3-4-5-1 a 5-cycle and 6 adjacent to the other five vertices. For edge $(1,6)$ to not be dominating, there must be a vertex 7 in $G$ adjacent to neither 1 nor 6. Then 7 is adjacent to 3 and 4 (nonneighbors of 1), and at least one of 2 and 5 since $G$ has no stable set of size 3. WLOG, 7 is adjacent to 2 and either (case 1) adjacent or (case 2) nonadjacent to 5.

(Case 1) For edge $(2,6)$ to not be dominating, $G$ must have a vertex 8 adjacent to neither 2 nor 6. Vertex 8 must be adjacent to 4, 5, and 7. If it is adjacent to 1, then $\{1,4,2,8,6,7\}$ induces $I_3$. If 8 is not adjacent to 1, then delete vertex 7 and witness that the resulting graph is isomorphic to the graph to be considered in case 2, so we reduce to that case.

(Case 2) For edge $(2,6)$ to not be dominating, $G$ must have a vertex 8 adjacent to neither 2 nor 6. Vertex 8 must be adjacent to 4, 5, and 7. If it is adjacent to 1, then $\{1,4,2,8,6,7\}$ induces $I_3$. If not, then it is adjacent to 3. For edge $(3,6)$ to not be dominating, $G$ must have a vertex 9 adjacent to neither 3 nor 6. Vertex 9 must be adjacent to 1, 5, 7, and 8. If 9 is adjacent to 2 then $\{2,5,3,9,6,8\}$ induces $I_3$. If it is not, then it is adjacent to 4 and $\{2,4,1,7,6,9\}$ induces $I_3$.

Therefore any minimal counterexample (and therefore every counterexample) to HC-$\{\overline{K_3}\}$ has an induced $I_3$.
\end{proof}

\begin{lemma}
HC-$\{\overline{K_3}, I_{14}\}$ holds.
\end{lemma}
\begin{proof}
By the previous lemma, any minimal counterexample $G$ HC-$\{\overline{K_3}\}$ has an induced $I_3$, say induced by vertices 1 through 6 with 1-2-3-4-5-6-1 a 6-cycle in the complement graph. For edge $(1,4)$ to not be dominating, $G$ must have a vertex 7 adjacent to neither 1 nor 4. But then since $G$ has no stable set of size 3, 7 must be adjacent to all four other vertices, and $\{1,2,3,4,5,6,7\}$ induces $I_{14}$.
\end{proof}

\begin{lemma}
HC-$\{\overline{K_3}, I_{21}\}$ holds.
\end{lemma}
\begin{proof}
Starting from the endpoint of the proof of the previous lemma, in order for $(2,5)$ to not be dominating, there must be a vertex 8 adjacent to neither 2 nor 5. Then 8 must be adjacent to all other vertices except possibly 7. If it is adjacent to 7 then $\{1,2,3,4,5,6,7,8\}$ induces $I_{21}$. If not, then for $(3,6)$ to not be dominating, there must be a vertex 9 adjacent to neither 3 nor 6; this vertex is adjacent to 1, 2, 3, 4, 5, 6, and at least one of 7 and 8. If 7 then $\{1,2,3,4,5,6,7,9\}$ induces $I_{21}$; if 8 then $\{1,2,3,4,5,6,8,9\}$ induces $I_{21}$ instead.
\end{proof}

This is probably the best result that can be obtained starting from Theorem~\ref{thm:past} using Algorithm~\ref{alg:dominating} alone without involving a tremendous amount of casework.

The graph $H_7$ in Figure~\ref{fig:h6h7} has no dominating edges, but we can using Algorithm~\ref{alg:fourcliqueadvanced} to introduce new vertices and dominating edges into this graph, and then use Algorithm~\ref{alg:dominating} afterwards. Here is an example, proving eventually HC-$\{\overline{K_3}, I_{12}\}$.
\begin{lemma}
HC-$\{\overline{K_3}, I_2\}$ holds.
\end{lemma}
\begin{proof}
By Theorem~\ref{thm:past}, any minimal counterexample $G$ to HC-$\{\overline{K_3}\}$ has an induced $H_7$, say induced by vertices 1 through 7 with 1-2-3-4-5-1 a 5-cycle, 6 adjacent to $\{1,2,3,7\}$, and 7 adjacent to $\{2,3,4,6\}$. Suppose for contradiction $G$ had no $I_2$.

Then the vertices $v$ of $G\setminus H_7$ have $N(v)\cap \{1,\dots, 7\}$ equal to one of the following sets:
\begin{center}
\begin{TAB}(c)[0.25em]{lll}{ccccccccc}
$N_{0}=\{3,4,5,7\}$, &
$N_{1}=\{3,4,5,6,7\}$, &
$N_{2}=\{2,3,4,6,7\}$, \\
$N_{3}=\{2,3,4,5,7\}$, &
$N_{4}=\{2,3,4,5,6,7\}$, &
$N_{5}=\{1,4,5,7\}$, \\
$N_{6}=\{1,4,5,6,7\}$, &
$N_{7}=\{1,4,5,6\}$, &
$N_{8}=\{1,4,5\}$, \\
$N_{9}=\{1,3,4,5,7\}$, &
$N_{10}=\{1,3,4,5,6,7\}$, &
$N_{11}=\{1,3,4,5,6\}$, \\
$N_{12}=\{1,3,4,5\}$, &
$N_{13}=\{1,2,5,6,7\}$, &
$N_{14}=\{1,2,5,6\}$, \\
$N_{15}=\{1,2,4,5,7\}$, &
$N_{16}=\{1,2,4,5,6,7\}$, &
$N_{17}=\{1,2,4,5,6\}$, \\
$N_{18}=\{1,2,4,5\}$, &
$N_{19}=\{1,2,3,6,7\}$, &
$N_{20}=\{1,2,3,5,6,7\}$, \\
$N_{21}=\{1,2,3,5,6\}$, &
$N_{22}=\{1,2,3,4,6,7\}$, &
$N_{23}=\{1,2,3,4,5,7\}$, \\
$N_{24}=\{1,2,3,4,5,6,7\}$, &
$N_{25}=\{1,2,3,4,5,6\}$, &
$N_{26}=\{1,2,3,4,5\}$.
\end{TAB}
\end{center}
(This is easily verified by brute-force; any other neighbor set leads either to a stable set of size 3 or an $I_2$.)

The set of vertices $S_i$ for which $N(v)\cap \{1,\dots, 7\}=N_i$ forms a clique in $G$ for all $i$ except possibly $i=24$ since $G$ has no stable set of size 3 and such vertices share a nonneighbor except in the case $i=24$. Suppose, for the moment, that $G$ had no induced $H_7\vee \overline{K_2}$. Then $S_{24}$ is also a clique in $G$.

For each pair $(i,j)$ with marked with an X in Table~\ref{tab:pairs}, $S_i$ and $S_j$ must be complete to each other in $G$, since if they weren't, $G$ would either have a stable set of size 3 or an $I_2$ (again, this can be verified by brute force by adding two vertices to $G$ with the appropriate neighbors and making those vertices nonadjacent, then witnessing that the resulting graph has a $\overline{K_3}$ or $I_2$).

\begin{table}[htb!]
    \centering
    \begin{TAB}(c)[0.25em]{c|cccccccccccccccccccccccccc}{c|cccccccccccccccccccccccccc}
         & 0 & 1 & 2 & 3 & 4 & 5 & 6 & 7 & 8 & 9 & 10 & 11 & 12 & 13 & 14 & 15 & 16 & 17 & 18 & 19 & 20 & 21 & 22 & 23 & 24 & 25 \\
        1  & X &  &  &  &  &  &  &  &  &  &  &  &  &  &  &  &  &  &  &  &  &  &  &  &  & \\
        2  & X & X &  &  &  &  &  &  &  &  &  &  &  &  &  &  &  &  &  &  &  &  &  &  &  & \\
        3  & X & X & X &  &  &  &  &  &  &  &  &  &  &  &  &  &  &  &  &  &  &  &  &  &  & \\
        4  & X & X & X & X &  &  &  &  &  &  &  &  &  &  &  &  &  &  &  &  &  &  &  &  &  & \\
        5  & X & X &  & X &  &  &  &  &  &  &  &  &  &  &  &  &  &  &  &  &  &  &  &  &  & \\
        6  & X & X &  &  & X & X &  &  &  &  &  &  &  &  &  &  &  &  &  &  &  &  &  &  &  & \\
        7  & X & X &  &  & X & X & X &  &  &  &  &  &  &  &  &  &  &  &  &  &  &  &  &  &  & \\
        8  & X & X &  & X &  & X & X & X &  &  &  &  &  &  &  &  &  &  &  &  &  &  &  &  &  & \\
        9  & X & X &  & X &  & X & X & X & X &  &  &  &  &  &  &  &  &  &  &  &  &  &  &  &  & \\
        10 & X & X &  &  & X & X & X & X & X & X &  &  &  &  &  &  &  &  &  &  &  &  &  &  &  & \\
        11 & X & X &  &  & X & X & X & X & X & X & X &  &  &  &  &  &  &  &  &  &  &  &  &  &  & \\
        12 & X & X &  & X &  & X & X & X & X & X & X & X &  &  &  &  &  &  &  &  &  &  &  &  &  & \\
        13 &  & X &  & X & X & X & X & X & X & X & X &  &  &  &  &  &  &  &  &  &  &  &  &  &  & \\
        14 &  &  &  &  &  & X & X & X & X &  &  & X & X & X &  &  &  &  &  &  &  &  &  &  &  & \\
        15 & X &  &  & X & X & X & X & X & X & X & X &  & X & X & X &  &  &  &  &  &  &  &  &  &  & \\
        16 &  & X &  & X & X & X & X & X & X & X & X & X &  & X & X & X &  &  &  &  &  &  &  &  &  & \\
        17 &  & X &  & X & X & X & X & X & X &  & X & X & X & X & X & X & X &  &  &  &  &  &  &  &  & \\
        18 & X &  &  & X & X & X & X & X & X & X &  & X & X & X & X & X & X & X &  &  &  &  &  &  &  & \\
        19 &  &  & X &  &  &  &  &  &  &  &  &  &  & X & X &  &  &  &  &  &  &  &  &  &  & \\
        20 &  & X &  & X & X & X & X &  &  & X & X & X & X & X & X & X & X &  &  & X &  &  &  &  &  & \\
        21 &  & X &  & X & X &  &  & X & X & X & X & X & X & X & X &  &  & X & X & X & X &  &  &  &  & \\
        22 &  &  & X &  &  & X & X & X &  & X & X & X & X &  &  & X & X & X & X & X &  &  &  &  &  & \\
        23 & X &  &  & X & X & X & X &  & X & X & X & X & X & X &  & X & X & X & X &  & X & X & X &  &  & \\
        24 &  & X &  & X & X & X & X & X &  & X & X & X & X & X &  & X & X & X & X &  & X & X & X & X &  & \\
        25 &  & X &  & X & X &  & X & X & X & X & X & X & X &  & X & X & X & X & X &  & X & X & X & X & X & \\
        26 & X &  &  & X & X & X &  & X & X & X & X & X & X &  & X & X & X & X & X &  & X & X & X & X & X & X
    \end{TAB}
    \caption{Pairs of sets $S_i$ and $S_j$ that must be complete to each other. Each such pair with $i>j$ is marked with an X in row $i$, column $j$.}
    \label{tab:pairs}
\end{table}

Then the sets given by
\begin{itemize}
\item $Q_1=\{4,5\} \cup S_{0} \cup S_{1} \cup S_{5} \cup S_{6} \cup S_{7} \cup S_{8} \cup S_{9} \cup S_{10} \cup S_{11} \cup S_{12},$
\item $Q_2=\{2,3\} \cup S_{3} \cup S_{4} \cup S_{20} \cup S_{21} \cup S_{23} \cup S_{24} \cup S_{25} \cup S_{26},$
\item $Q_3=\{2,3,6,7\} \cup S_{2} \cup S_{19} \cup S_{22},$ and
\item $Q_4=\{1,2\} \cup S_{13} \cup S_{14} \cup S_{15} \cup S_{16} \cup S_{17} \cup S_{18}$
\end{itemize}
are a clique cover of $G$, and $\abs{Q_1}+\abs{Q_2}+\abs{Q_3}+\abs{Q_4}\ge \abs{V(G)}+2$, so $\omega(G)\ge \frac{\abs{V(G)}+2}{4}$, so by Lemma~\ref{lem:fourclique}, $G$ is not a counterexample to HC-$\{\overline{K_3}\}$. This means our supposition that $G$ was $(H_7\vee \overline{K_2})$-free was incorrect, so $G$ has a vertex 8 and 9 both adjacent to $\{1,\dots,7\}$ but not adjacent to each other.

Then for $(1,8)$ to not be a dominating edge, $G$ must have a vertex 10 adjacent to neither 1 nor 8. Vertex 10 must be adjacent to 3 and 4. If 10 is adjacent to 2 or 5 then $\{9,8,10,1,4,2\}$ or $\{9,8,10,1,3,5\}$ induces $I_2$, respectively. But 10 cannot be nonadjacent to both 2 and 5, so we have a contradiction and $G$ does have an $I_2$.
\end{proof}

The proof above essentially used Algorithm~\ref{alg:fourcliqueadvanced} with $k=1$, finishing by eliminating a dominating edge. As can be seen, this clique cover technique is computationally intensive and rather impractical to verify by hand, even for this relatively small example.

\begin{lemma}
HC-$\{\overline{K_3}, I_{12}\}$ holds.
\end{lemma}
\begin{proof}
From the previous lemma, any minimal counterexample $G$ of HC-$\{\overline{K_3}\}$ has an induced $I_2$, say by vertices 1 through 6 with 1-2-3-4-5-6 a 6-path in the complement graph. For $(2,5)$ to not be dominating, $G$ must have a vertex 7 adjacent to neither 2 nor 5, but this vertex must be adjacent to 1, 3, 4, and 6, so $\{2,7,5,4,3,1,6\}$ induces $I_{12}$.
\end{proof}

\subsection{Data for Computer Verification}

As we discussed in Section~\ref{sec:behavior}, in Algorithm~\ref{alg:dominating} and Algorithm~\ref{alg:full}, one needs to choose a dominating edge whenever a graph has multiple. We describe our method for this in more detail in Section~\ref{sec:behavior} and Appendix~\ref{app:implementation}. The vertex labeling of the initial graph $H$ is given by the labelings in Figures~\ref{fig:main} and~\ref{fig:intermediate}. Whenever a vertex is added to a graph, it gets the smallest positive integer label not yet used in that graph. The chosen dominating edge is the one that minimizes the larger label of its endpoints.

We then present the necessary steps to verify Theorem~\ref{thm:main} as a directed forest. Each edge $G\to H$ means to run Algorithm~\ref{alg:full} starting from graph $G$ to prove HC-$\{\overline{K_3}, H\}$, i.e. it is the assertion $G\vdash_\Ac H$ for our implementation $\Ac$ of Algorithm~\ref{alg:full}.

There are four starting points used: $W_5$, $\overline{K_{2,3}}$, $K_7$, and $H_7$. All four are known to be induced subgraphs of any counterexample to HC-$\{\overline{K_3}\}$ due to Theorem~\ref{thm:past}. The labeling to use for $W_5$ is to label the vertices in the 5-cycle as 1-2-3-4-5 and the hub as 6. The labeling to use for $\overline{K_{2,3}}$ is that the part with three vertices gets labels 1, 3, and 5 and the part with two vertices gets labels 2 and 4. All labelings of $K_7$ are equivalent due to symmetry. The labeling to use for $H_7$ is to label a 5-cycle 1-2-3-4-5 and vertex 6 is adjacent to 1, 2, 3, 7 and vertex 7 is adjacent to 2, 3, 4, 7.

The directed forest is shown in Figure~\ref{fig:data}. There are many other possible connections between the graphs in the diagram; the arrows we display seek to minimize the time it takes for a computer to verify each step. We describe this in more detail in Section~\ref{sec:conclusion}, as well as explain briefly how this directed forest was found starting from Theorem~\ref{thm:past}.

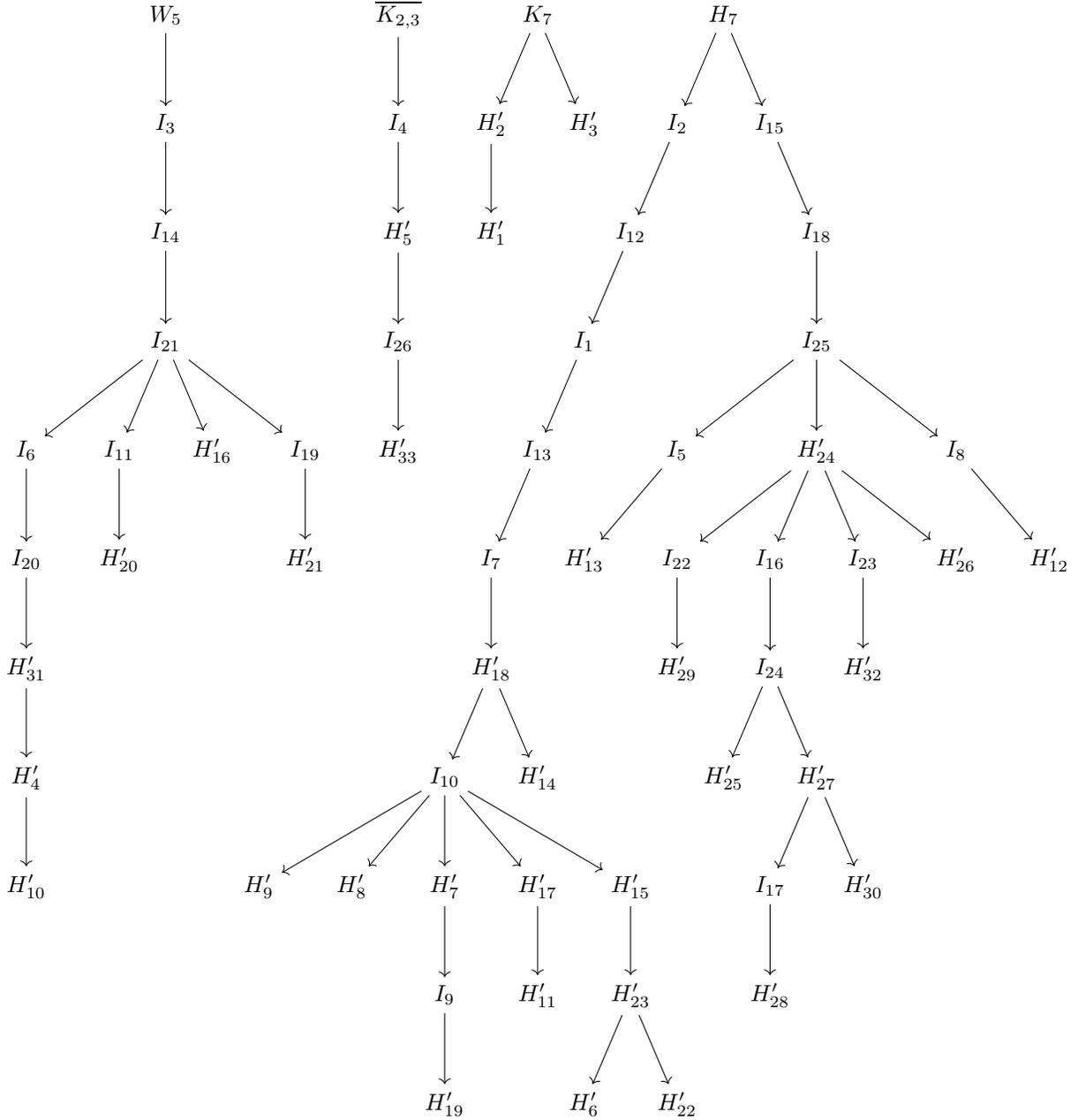
\begin{figure}[htbp!]
    \centering
    \begin{tikzpicture}[yscale=1.64,xscale=0.70]
\node (K23c) at (5, 0) {$\overline{K_{2,3}}$};
\node (I4) at (5, -1) {$I_{4}$};
\node (K7) at (8, 0) {$K_7$};
\node (H2p) at (7, -1) {$H_{2}'$};
\node (W5) at (0, 0) {$W_5$};
\node (I3) at (0, -1) {$I_{3}$};
\node (H7) at (12, 0) {$H_7$};
\node (I15) at (13, -1) {$I_{15}$};
\node (H3p) at (9, -1) {$H_{3}'$};
\node (I2) at (11, -1) {$I_{2}$};
\node (I12) at (10, -2) {$I_{12}$};
\node (I18) at (14, -2) {$I_{18}$};
\node (I14) at (0, -2) {$I_{14}$};
\node (H1p) at (7, -2) {$H_{1}'$};
\node (H5p) at (5, -2) {$H_{5}'$};
\node (I25) at (14, -3) {$I_{25}$};
\node (I26) at (5, -3) {$I_{26}$};
\node (I21) at (0, -3) {$I_{21}$};
\node (I1) at (9, -3) {$I_{1}$};
\node (I13) at (8, -4) {$I_{13}$};
\node (I5) at (11, -4) {$I_{5}$};
\node (H33p) at (5, -4) {$H_{33}'$};
\node (H24p) at (14, -4) {$H_{24}'$};
\node (I6) at (-3, -4) {$I_{6}$};
\node (I8) at (17, -4) {$I_{8}$};
\node (I11) at (-1, -4) {$I_{11}$};
\node (H16p) at (1, -4) {$H_{16}'$};
\node (I19) at (3, -4) {$I_{19}$};
\node (H21p) at (3, -5) {$H_{21}'$};
\node (H13p) at (9, -5) {$H_{13}'$};
\node (I7) at (7, -5) {$I_{7}$};
\node (H20p) at (-1, -5) {$H_{20}'$};
\node (I22) at (11, -5) {$I_{22}$};
\node (I16) at (13, -5) {$I_{16}$};
\node (I23) at (15, -5) {$I_{23}$};
\node (H26p) at (17, -5) {$H_{26}'$};
\node (I20) at (-3, -5) {$I_{20}$};
\node (H12p) at (19, -5) {$H_{12}'$};
\node (H18p) at (7, -6) {$H_{18}'$};
\node (H32p) at (15, -6) {$H_{32}'$};
\node (H31p) at (-3, -6) {$H_{31}'$};
\node (H29p) at (11, -6) {$H_{29}'$};
\node (I24) at (13, -6) {$I_{24}$};
\node (H4p) at (-3, -7) {$H_{4}'$};
\node (H25p) at (12, -7) {$H_{25}'$};
\node (H14p) at (8, -7) {$H_{14}'$};
\node (H27p) at (14, -7) {$H_{27}'$};
\node (I10) at (6, -7) {$I_{10}$};
\node (H7p) at (6, -8) {$H_{7}'$};
\node (H17p) at (8, -8) {$H_{17}'$};
\node (I17) at (13, -8) {$I_{17}$};
\node (H15p) at (10, -8) {$H_{15}'$};
\node (H8p) at (4, -8) {$H_{8}'$};
\node (H30p) at (15, -8) {$H_{30}'$};
\node (H10p) at (-3, -8) {$H_{10}'$};
\node (H9p) at (2, -8) {$H_{9}'$};
\node (H23p) at (10, -9) {$H_{23}'$};
\node (H28p) at (13, -9) {$H_{28}'$};
\node (H11p) at (8, -9) {$H_{11}'$};
\node (I9) at (6, -9) {$I_{9}$};
\node (H6p) at (9, -10) {$H_{6}'$};
\node (H19p) at (6, -10) {$H_{19}'$};
\node (H22p) at (11, -10) {$H_{22}'$};
\draw[->] (K23c) -- (I4);
\draw[->] (K7) -- (H2p);
\draw[->] (W5) -- (I3);
\draw[->] (H7) -- (I15);
\draw[->] (K7) -- (H3p);
\draw[->] (H7) -- (I2);
\draw[->] (I2) -- (I12);
\draw[->] (I15) -- (I18);
\draw[->] (I3) -- (I14);
\draw[->] (H2p) -- (H1p);
\draw[->] (I4) -- (H5p);
\draw[->] (I18) -- (I25);
\draw[->] (H5p) -- (I26);
\draw[->] (I14) -- (I21);
\draw[->] (I12) -- (I1);
\draw[->] (I1) -- (I13);
\draw[->] (I25) -- (I5);
\draw[->] (I26) -- (H33p);
\draw[->] (I25) -- (H24p);
\draw[->] (I21) -- (I6);
\draw[->] (I25) -- (I8);
\draw[->] (I21) -- (I11);
\draw[->] (I21) -- (H16p);
\draw[->] (I21) -- (I19);
\draw[->] (I19) -- (H21p);
\draw[->] (I5) -- (H13p);
\draw[->] (I13) -- (I7);
\draw[->] (I11) -- (H20p);
\draw[->] (H24p) -- (I22);
\draw[->] (H24p) -- (I16);
\draw[->] (H24p) -- (I23);
\draw[->] (H24p) -- (H26p);
\draw[->] (I6) -- (I20);
\draw[->] (I8) -- (H12p);
\draw[->] (I7) -- (H18p);
\draw[->] (I23) -- (H32p);
\draw[->] (I20) -- (H31p);
\draw[->] (I22) -- (H29p);
\draw[->] (I16) -- (I24);
\draw[->] (H31p) -- (H4p);
\draw[->] (I24) -- (H25p);
\draw[->] (H18p) -- (H14p);
\draw[->] (I24) -- (H27p);
\draw[->] (H18p) -- (I10);
\draw[->] (I10) -- (H7p);
\draw[->] (I10) -- (H17p);
\draw[->] (H27p) -- (I17);
\draw[->] (I10) -- (H15p);
\draw[->] (I10) -- (H8p);
\draw[->] (H27p) -- (H30p);
\draw[->] (H4p) -- (H10p);
\draw[->] (I10) -- (H9p);
\draw[->] (H15p) -- (H23p);
\draw[->] (I17) -- (H28p);
\draw[->] (H17p) -- (H11p);
\draw[->] (H7p) -- (I9);
\draw[->] (H23p) -- (H6p);
\draw[->] (I9) -- (H19p);
\draw[->] (H23p) -- (H22p);
    \end{tikzpicture}
    \caption{The directed forest for verifying Theorem~\ref{thm:main}.}
    \label{fig:data}
\end{figure}

\section{Final Remarks}
\label{sec:conclusion}


Call the \textit{weight} (relative to an implementation $\Ac$ of Algorithm~\ref{alg:full}) of input $(H, G)$ the sum over the graphs $H'$ added to $A$ at any point in the algorithm (including $H$) of $2^{\abs{V(H')}}$, assuming $H\vdash_\Ac G$ so that this weight is finite. This measure is designed to give a very rough estimate of the amount of time it takes for Algorithm~\ref{alg:full} to prove HC-$\{\overline{K_3}, G\}$ starting from $H$. This is because there are several parts of the algorithm that scale exponentially (or at least super-polynomially) in the size of the graph $H'$ considered at some point:
\begin{itemize}
    \item When one constructs graphs by adding a vertex to $H'$ while either avoiding or introducing a dominating edge, there may be exponentially many such graphs.
    \item In several places we must solve the subgraph isomorphism problem, which is NP-complete in general. Even though this is only polynomial-time for fixed $G$, recall that in some places we check if graphs are $A$-free with $A$ containing possibly arbitrarily large graphs, so we are still potentially subject to the full NP-completeness of this problem.
    \item Running Algorithm~\ref{alg:fourcliqueadvanced} requires possibly constructing an auxiliary graph with exponentially many vertices. Worse, we then try to find a clique cover of this graph with four cliques, which is equivalent to 4-coloring the complement graph and is thus NP-complete in general. The total running time of Algorithm~\ref{alg:fourcliqueadvanced} is therefore actually potentially doubly-exponential in the number of vertices of the input.
\end{itemize}

Weight the directed edges of Figure~\ref{fig:data} according to this metric. The proof data expressed in Figure~\ref{fig:data} is optimized to minimize the largest weight in any individual edge from a graph $H_i'$ to one of the roots using our specific implementation of Algorithm~\ref{alg:full}. In fact, the way these proofs were found was by running in parallel all inputs $(H, G)$ with $H$ a maximal graph for which HC-$\{\overline{K_3}, H\}$ is known and $G$ a minimal graph that is not an induced subgraph of any graph $H'$ for which HC-$\{\overline{K_3}, H'\}$ is known, where by ``parallel'' we mean at each point in time to step forward whichever instance currently has the lowest running total weight. In this way, when an instance returns \textsc{success}, we can be sure it is the most efficient possible improvement.

This was performed starting from the graphs in Theorem~\ref{thm:past} until no new improvements were found with a weight less than $2.3\cdot 10^{7}$. There are 69 minimal graphs that are not a subgraph of any $H_i'$ or $K_7$ or $H_7$, though several of these (such as $\overline{K_{3,3}}$) could not possibly be proved by the methods in this paper due to Theorem~\ref{thm:corefailure}. We expect if one continues to run this process, one will continue to find new graphs for which HC-$\{\overline{K_3}, H\}$ holds for a long time, though the improvements become less frequent over time.

The weights for each edge in Figure~\ref{fig:data} are given in Appendix~\ref{app:moredata}.

There are several routes that could improve the results obtained in this paper. First, an mentioned in a footnote in the introduction, there is a recent improvement to Lemma~\ref{lem:conings} \cite{trianglefork}, so one can replace the tree $T_1$ with the larger $T_1'$ obtained by adding a second leaf attached to the third vertex of the 5-path of $T_1$, and replace $F_3$ accordingly. We expect perhaps a few more results can be obtained with this substitution.

Second, Lemma~\ref{lem:dominating} has a stronger version:
\begin{lemma}\label{lem:matching}\cite{pst}
Any minimal counterexample to HC-$\{\overline{K_3}\}$ has no connected dominating matching.
\end{lemma}
Here, a \textit{connected matching} is a set of vertex-disjoint edges that are pairwise adjacent. It is \textit{dominating} if additionally every vertex outside the matching is adjacent to every edge in the matching. A dominating edge is just a connected dominating matching with one edge. This substitution theoretically allows you to sidestep Theorem~\ref{thm:corefailure}, especially the restriction on $\overline{C_4}$-cores; for instance it was the main ingredient used to prove that HC-$\{\overline{K_3},H_7\}$ holds in \cite{pst}. It is clear how to extend Algorithm~\ref{alg:dominating} to incorporate this stronger result. However, it is not immediately clear how to reconcile it with Algorithm~\ref{alg:fourclique}/\ref{alg:fourcliqueadvanced}, since it is conjectured:

\begin{conjecture}\label{conj:ssh}
All $\overline{K_3}$-free graphs on $n$ vertices have a set of $\ceil{n/2}$ pairwise adjacent vertices or edges.
\end{conjecture}

This conjecture is called \textit{Seymour's Strengthening} of HC-$\{\overline{K_3}\}$ in \cite{blasiak} or the \textit{matching version} of HC-$\{\overline{K_3}\}$ in \cite{packingseagulls}. It is very close to conjecturing that all $\overline{K_3}$-free graphs have a connected dominating matching (which is a slightly stronger statement), so at the very least one should expect ``most'' $\overline{K_3}$-free graphs to have a connected dominating matching. Thus, na\"ively replacing the check for dominating edges with a check for connected dominating matchings likely completely removes the impact of Algorithms~\ref{alg:fourclique} and~\ref{alg:fourcliqueadvanced}.

Third, if one marks vertices in the graph and only checks for subgraphs containing the set of marked vertices, rather than subgraphs in the whole graph, one can prove theorems of the form ``In any minimal counterexample to HC-$\{\overline{K_3}\}$, every induced $G$ must be part of an induced $G'$'' where $G'$ contains $G$ as an induced subgraph. For example:

\begin{theorem}
In any minimal counterexample to HC-$\{\overline{K_3}\}$, every vertex must be found in some induced $C_5$.
\end{theorem}
\begin{proof}
Suppose $G$ is a minimal counterexample to HC-$\{\overline{K_3}\}$. Let $a$ be an arbitrary vertex to $G$. Obviously $a$ must have a nonneighbor $c$, and $G$ must be connected so we may take $c$ to be adjacent to a neighbor of $a$, say $b$. The edge $ab$ is not a dominating edge, so there is a vertex $d$ adjacent to neither $a$ nor $b$; this vertex must be adjacent to $c$ to avoid $\overline{K_3}$. Then for $bc$ to not be a dominating edge, there is a vertex $e$ adjacent to neither $b$ nor $c$; this vertex must be adjacent to $a$ and $d$. Then $\{a,b,c,d,e\}$ is an induced $C_5$ containing $a$.
\end{proof}

In a future paper, we will modify the algorithms in this paper to obtain many more results like the one above and also incorporate the stronger versions of Lemmas~\ref{lem:dominating} and~\ref{lem:conings}.

\section{Acknowledgements}

This work was supported by funding from Princeton University for undergraduate independent projects. Thank you to Maria Chudnovsky, Paul Seymour, Zach Hunter, and Jan Goedgebeur for helpful comments on the paper and work.

\bibliographystyle{alpha}
\bibliography{biblio}

\begin{thebibliography}{BGSP12}

\bibitem[BGSP12]{ramsey}
Gunnar Brinkmann, Jan Goedgebeur, and Jan-Christoph Schlage-Puchta.
\newblock Ramsey numbers {R}({K}3,{G}) for graphs of order 10.
\newblock {\em The Electronic Journal of Combinatorics}, 19(4), 2012.

\bibitem[Bla07]{blasiak}
Jonah Blasiak.
\newblock A special case of {H}adwiger's conjecture.
\newblock {\em Journal of Combinatorial Theory, Series B}, 97(6):1056--1073,
  2007.

\bibitem[Bos19]{bosse}
Christian Bosse.
\newblock A note on {H}adwiger’s conjecture for {$W_5$}-free graphs with
  independence number two.
\newblock {\em Discrete Mathematics}, 342, 2019.

\bibitem[CN06]{igraph}
Gabor Csardi and Tamas Nepusz.
\newblock The igraph software package for complex network research.
\newblock {\em InterJournal}, Complex Systems:1695, 2006.

\bibitem[CS12]{packingseagulls}
Maria Chudnovsky and Paul Seymour.
\newblock Packing seagulls.
\newblock {\em Combinatorica}, 32:251--282, 2012.

\bibitem[HSS08]{networkx}
Aric~A. Hagberg, Daniel~A. Schult, and Pieter~J. Swart.
\newblock {Exploring Network Structure, Dynamics, and Function using NetworkX}.
\newblock In Ga\"el Varoquaux, Travis Vaught, and Jarrod Millman, editors, {\em
  Proceedings of the 7th Python in Science Conference}, pages 11 -- 15,
  Pasadena, CA USA, 2008.

\bibitem[IMM18]{pysat}
Alexey Ignatiev, Antonio Morgado, and Joao Marques{-}Silva.
\newblock {PySAT:} {A} {Python} toolkit for prototyping with {SAT} oracles.
\newblock In {\em SAT}, pages 428--437, 2018.

\bibitem[JR95]{trianglefree5chromatic}
Tommy Jensen and Gordon Royle.
\newblock Small graphs with chromatic number 5: A computer search.
\newblock {\em Journal of Graph Theory}, 19(1):107--116, 1995.

\bibitem[Kri10]{kriesell}
Matthias Kriesell.
\newblock On {S}eymour's strengthening of {H}adwiger's conjecture for graphs
  with certain forbidden subgraphs.
\newblock {\em Discrete Mathematics}, 310:2714--2724, 2010.

\bibitem[McK]{ramseygraphs}
Brendan McKay.
\newblock Combinatorial data.
\newblock \url{http://users.cecs.anu.edu.au/~bdm/data/ramsey.html}.

\bibitem[PST03]{pst}
Michael~D. Plummer, Michael Stiebitz, and Bjarne Toft.
\newblock On a special case of {H}adwiger's conjecture.
\newblock {\em Discussiones Mathematicae Graph Theory}, 23:333--363, 2003.

\bibitem[Ran04]{colorablepairs}
Bert Randerath.
\newblock 3-{C}olorability and forbidden subgraphs. {I}: {C}haracterizing
  pairs.
\newblock {\em Discrete Mathematics}, 276(1):313--325, 2004.
\newblock 6th International Conference on Graph Theory.

\bibitem[Sey16]{seymoursurvey}
Paul Seymour.
\newblock Hadwiger's conjecture.
\newblock In {\em Open Problems in Mathematics}, pages 417--437, 2016.

\bibitem[SWY21]{trianglefork}
Joshua Schroeder, Zhiyu Wang, and Xingxing Yu.
\newblock On the 3-colorability of triangle-free and fork-free graphs, 2021.

\end{thebibliography}

\newpage
\appendix

\section{Further Implementation Details and Code}
\label{app:implementation}

Graph computations were done using the igraph Python package (\cite{igraph}). In many places one must check if a graph has a particular induced subgraph (either $\overline{K_3}$ for the independence number 3 checks, or some given graph $H$ for the the property $P$ checks whenever $P$ is the property of containing an induced $H$). This problem is NP-complete, but it can be solved quickly enough for small graphs such as ours. In our case, we used igraph's implementation of the LAD algorithm.

To perform the check for a four-clique cover of $G$ at the end of Algorithm~\ref{alg:fourclique}, we model it as a \textsc{sat} instance and use a \textsc{sat} solver. Specifically, there are four variables $v_1$, $v_2$, $v_3$, $v_4$ for each vertex $v$ in $G$ corresponding to the four cliques that $v$ can be included in. The boolean formula is the conjunction of the following clauses:
\begin{itemize}
    \item $(v_1\vee v_2\vee v_3\vee v_4)$ for each $v\in V(G)$, ensuring that $Q_1,Q_2,Q_3,Q_4$ cover $G$;
    \item $(\neg v_i\vee \neg u_i)$ for each $uv\in E(G)$ and $i=1,2,3,4$, ensuring that $Q_1,Q_2,Q_3,Q_4$ are cliques;
    \item $(\sum_{i=1}^4\sum_{v\in V(H)}v_i\ge \abs{H}+2)$, ensuring the additional condition $\sum_{i=1}^4 \abs{Q_i\cap V(H)}\ge \abs{H}+2$.
\end{itemize}
We used PySAT (\cite{pysat}) to solve these \textsc{sat} instances.

Also, often one needs to generate the set of graphs formed by adding one vertex $v$ to a graph $H$ with some restrictions on the result. Instead of trying to add $v$ to $H$, we instead start with $v$ and then add the vertices of $H$ one at a time, at each step checking the imposed restrictions. This allows one to ``bail out early,'' so to speak, if a potential attachment does not satisfy the restrictions. Some care must be taken afterwards to fix the vertex labelings so that $H$ gets labels 1 through $\abs{V(H)}$ and $v$ gets label $\abs{V(H)}+1$.

Finally, much computation can be reused when running Algorithm~\ref{alg:fourcliqueadvanced} with the same input aside from $k$. In fact, the sets $A$, $B$, and $C$ are the same regardless of $k$, so we temporarily cache them whenever Algorithm~\ref{alg:full} reaches the part where Algorithm~\ref{alg:fourcliqueadvanced} should be run.

We provide code at \href{https://github.com/dcartermath/hc-forbidden-subgraphs}{https://github.com/dcartermath/hc-forbidden-subgraphs} which can be used to verify Theorem~\ref{thm:main}. We used Python version 3.9.6, igraph version 0.9.11, and PySAT version 0.1.7.dev19, though the code most likely works fine with newer versions. The program reads the TeX source of Figures~\ref{fig:main},~\ref{fig:intermediate}, and~\ref{fig:data}, respectively copied into files \verb|hp_figure|, \verb|i_figure|, and \verb|edges_figure| put in the same directory as this program, to determine the graphs, labelings, and steps to verify Theorem~\ref{thm:main}, thus making absolutely sure that the data provided in the main body of this paper is correct. Each line of output is one arrow of Figure~\ref{fig:data} consisting of the name of the input graph, the name of the output graph, the total number of graphs ever added to the set $A$, and the weight.

In the process of discovering Theorem~\ref{thm:main}, we did not immediately use the optimizations discussed above. Initially, we simply brute-force checked for four-clique covers with the desired properties, and computed the set of graphs formed by adding a single vertex to a given graph $H$ by simply trying all $2^{\abs{V(H)}}$ possibilities for the neighbor set of the new vertex. Whenever we made an optimization, we verified the newer version gave an identical state to the older version run for one hour. Thus, we are confident in the correctness of our implementation.

\section{More Data}
\label{app:moredata}

Running the code in Appendix~\ref{app:implementation}, after some minutes or hours of computation, produces the following output:
\begin{verbatim}
K23c I4 2 544
K7 H2p 121 317568
W5 I3 5 832
H7 I15 2 640
K7 H3p 600 6222976
H7 I2 2 640
I2 I12 4 576
I15 I18 27 194688
I3 I14 1 64
H2p H1p 532 2454144
I4 H5p 26 423488
I18 I25 1 128
H5p I26 65 324736
I14 I21 2 384
I12 I1 7 5248
I1 I13 2 192
I25 I5 336 3440896
I26 H33p 348 20723968
I25 H24p 134 715008
I21 I6 20 42752
I25 I8 76 527616
I21 I11 33 102144
I21 H16p 139 768768
I21 I19 35 84736
I19 H21p 121 336128
I5 H13p 524 7547520
I13 I7 17 19584
I11 H20p 73 407168
H24p I22 522 3486976
H24p I16 95 604928
H24p I23 989 22616320
H24p H26p 136 611072
I6 I20 32 28800
I8 H12p 131 411776
I7 H18p 112 345216
I23 H32p 1 256
I20 H31p 45 1067776
I22 H29p 333 6226688
I16 I24 1 128
H31p H4p 338 5561856
I24 H25p 1348 15871232
H18p H14p 566 14731520
I24 H27p 94 621824
H18p I10 265 1639680
I10 H7p 256 1474432
I10 H17p 218 1423232
H27p I17 354 8121088
I10 H15p 265 11495296
I10 H8p 756 10430336
H27p H30p 327 9972480
H4p H10p 899 10312832
I10 H9p 305 1976192
H15p H23p 479 15760640
I17 H28p 1 128
H17p H11p 341 1031424
H7p I9 369 3561216
H23p H6p 723 13390592
I9 H19p 19 8320
H23p H22p 74 17303296
\end{verbatim}
The four columns here separated by spaces are the in-vertex $G$ and out-vertex $H$ of an arrow of Figure~\ref{fig:data}, then the total number of graphs added to $A$ when Algorithm~\ref{alg:full} is run with input $(G, H)$, and finally the weight of the arrow, which is the sum over $H'$ ever added to $A$ of $2^{\abs{V(H')}}$.

\section{Proving HC-\texorpdfstring{$\{\overline{K_3},K_8\}$}{\{K3, K8\}}}
\label{app:k8}

As we mentioned in the introduction, Theorem~\ref{thm:main} is no help for proving Theorem~\ref{thm:k8} because all 477142 $\{\overline{K_3},K_8\}$-free graph on 27 vertices contains all $H_i'$, and $K_7$ and $H_7$. However, the vast majority of these graphs have a dominating edge, and it turns out that the ones that don't have a connected dominating matching with two edges. Specifically:
\begin{proposition}\label{prop:k8specific}
Of the 477142 $\{\overline{K_3},K_8\}$-free graph on 27 vertices, 455344 have a dominating edge and the remaining 21798 have a connected dominating matching of two edges.
\end{proposition}
Code to verify this is available at \href{https://github.com/dcartermath/hc-forbidden-subgraphs}{https://github.com/dcartermath/hc-forbidden-subgraphs}. This code, \verb|k8.py|, simply checks each graph in turn, using the file containing such graphs available at \cite{ramseygraphs} (though note that the relevant file contains all $\{K_3,\overline{K_8}\}$-free graph on 27 vertices, so we must take the complement of them). The graph data is parsed using NetworkX version 2.8.3 (\cite{networkx}), but the graphs are immediately converted to igraph. The program outputs several lines of intermediate computation before eventually the line \verb|477142 / 477142 [455344, 21798]|, indicating that all 477142 graphs were read, 455344 have a dominating edge, and the remaining 21798 have a connected dominating matching of two edges.

The check for connected dominating matchings of two edges is simply hard-coded in; a pair of edges $uv,xy$ is a connected dominating matching if all of the following conditions are met:
\begin{itemize}
    \item $\{u,v\}$ and $\{x,y\}$ are disjoint.
    \item $\{u,v\}\cap N(x)\ne \varnothing$ or $\{u,v\}\cap N(y)\ne \varnothing$.
    \item For each $z\in V(G)\setminus \{u,v,x,y\}$, $z\in N(u)$ or $z\in N(v)$, and $z\in N(x)$ or $z\in N(y)$.
\end{itemize}

Theorem~\ref{thm:k8} follows from Proposition~\ref{prop:k8specific} due to Lemma~\ref{lem:matching} and the fact that HC-$\{\overline{K_3},K_7\}$ holds. The latter fact implies (similarly to the reasoning in Corollary~\ref{cor:k8bound}) that any counterexample to HC-$\{\overline{K_3}\}$ has at least 27 vertices. Thus the graphs considered in Proposition~\ref{prop:k8specific} either satisfy Hadwiger's Conjecture or are minimal counterexamples.

\end{document}